\documentclass[11pt]{article}

\usepackage{tabularx}

\usepackage[english]{babel}
\usepackage{inputenc}
\usepackage{csquotes}   
\usepackage[T1]{fontenc}
\usepackage[usenames,dvipsnames,svgnames,table]{xcolor}	
\usepackage{soul}							
\usepackage{paralist}							   
\usepackage[shortlabels]{enumitem}        

\usepackage{verbatim}                           
\usepackage{spverbatim} 
\usepackage{listings}                              

\usepackage{graphicx}                            
\usepackage[tight]{subfigure}                 
\usepackage{rotating}								
\usepackage[footnotesize]{caption}					
\usepackage[table]{xcolor}
\usepackage{adjustbox}			

\usepackage{mathtools}                         
\usepackage{amsfonts}                          
\usepackage{amssymb}                          
\usepackage{dsfont,mathrsfs}                 
\usepackage{bm}										
\usepackage{slantsc}		
\usepackage[makeroom]{cancel}			
\usepackage{bigints}    

\usepackage{titletoc}					
\usepackage[nottoc,notbib]{tocbibind}	
\RequirePackage{secdot}                   
\usepackage{appendix}					

\usepackage{amsthm}              
\usepackage{xspace}                
\usepackage{pdflscape}                 
\allowdisplaybreaks                         
\usepackage[shortcuts]{extdash}    
\usepackage{ifpdf}                         
\usepackage{ifthen}				
\usepackage{microtype}          
\usepackage{pdfpages}			

\usepackage[paper=a4paper,top=2cm,bottom=2.2cm,inner=2cm,outer=2cm]{geometry}

\usepackage{setspace}				
\usepackage[hidelinks=true]{hyperref}

\usepackage[nottoc]{tocbibind}

\usepackage[nameinlink]{cleveref}		

\newtheorem{dummy}{Dummy}[section]              

\Crefname{proposition}{Proposition}{Propositions}
\newtheorem{lemma}[dummy]{Lemma}
\Crefname{lemma}{Lemma}{Lemmas}
\newtheorem{theorem}[dummy]{Theorem}
\Crefname{theorem}{Theorem}{Theorems}

\theoremstyle{definition}

\newtheorem{remark}[dummy]{Remark}

\newtheorem{assumption}{Assumption}
\Crefname{assumption}{Assumption}{Assumptions}

\newcommand{\e}[1][]{\ensuremath{\mathds{E}_{#1}}\xspace} 
\newcommand{\pr}[1][]{\ensuremath{\mathds{P}_{#1}}\xspace} 
\newcommand{\abs}[1][\cdot]{\ensuremath{\left| {#1} \right|}\xspace}        
\newcommand{\distance}[1][]{\ensuremath{d(#1)}\xspace}   
\newcommand{\supnorm}[1][\cdot]{\ensuremath{ \| #1 \| }\xspace}        
\newcommand{\borel}{\ensuremath{ \mathcal{B} }\xspace} 
\newcommand{\realborel}{\ensuremath{ \borel(\reals) }\xspace} 
\newcommand{\borelset}{\ensuremath{ B }\xspace} 
\newcommand{\vect}[1]{\ensuremath{\boldsymbol{{#1}} }\xspace} 
\newcommand{\indicatorfunction}[1][]{\ensuremath{\mathds{1}_{#1}}\xspace} 
\newcommand{\uniform}[1][i]{\ensuremath{U_{#1}}\xspace} 
\newcommand{\smallnumber}{\ensuremath{\delta}\xspace} 
\newcommand{\closeness}[2][\smallnumber]{\ensuremath{\prescript{}{#1}{#2}}\xspace} 
\newcommand{\uniformdistribution}{\ensuremath{\mathcal{U}}\xspace} 

\newcommand{\naturals}[1][]{\ensuremath{\mathds{N}_{#1}}\xspace} 
\newcommand{\reals}[1][]{\ensuremath{\mathds{R}^{#1}}\xspace} 
\newcommand{\skorokhodspace}[1][]{\ensuremath{\mathds{D}^{#1}}\xspace} 
\newcommand{\skorokhodtopology}{\ensuremath{J_1}\xspace} 
\newcommand{\increasingstepset}{\ensuremath{\mathds{D}^{\uparrow}_{\mathcal{S}}}\xspace} 
\newcommand{\bijectionset}{\ensuremath{\Lambda}\xspace} 

\newcommand{\claimsizesymbol}{\ensuremath{X}\xspace} 
\newcommand{\claim}[1][i]{\ensuremath{\claimsizesymbol_{#1}}\xspace} 
\newcommand{\claimsizedissymbol}{\ensuremath{F}\xspace} 
\newcommand{\claimsizedis}[1][x]{\ensuremath{\claimsizedissymbol(#1)}\xspace} 
\newcommand{\comclaimsizedis}[1][x]{\ensuremath{\bar{\claimsizedissymbol}(#1)}\xspace} 
\newcommand{\poissonprocess}[1][t]{\ensuremath{N(#1)}\xspace} 
\newcommand{\poissonrate}{\ensuremath{\lambda}\xspace} 
\newcommand{\orderedclaim}[2][i]{\ensuremath{\claimsizesymbol^\star_{#1,\poissonprocess[#2]}}\xspace} 
\newcommand{\aggregateclaims}[1][t]{\ensuremath{S(#1)}\xspace} 
\newcommand{\scaledaggregateclaims}[2][t]{\ensuremath{\bar{S}_{#2}(#1)}\xspace} 
\newcommand{\scaledaggregateclaimssequence}[1][n]{\ensuremath{\bar{S}_{#1}}\xspace} 
\newcommand{\cedentpremium}{\ensuremath{p_D}\xspace} 
\newcommand{\reinsurerpremium}{\ensuremath{p_R}\xspace} 
\newcommand{\premium}{\ensuremath{p}\xspace} 
\newcommand{\initialcapital}{\ensuremath{u}\xspace} 
\newcommand{\finiteruinprobability}[1][T]{\ensuremath{\psi(\initialcapital,#1)}\xspace} 
\newcommand{\asymptotfiniteruinprobability}[1][\crosslevel]{\ensuremath{\psi(n #1,n)}\xspace} 

\newcommand{\randomset}{\ensuremath{A}\xspace} 
\newcommand{\openset}[1][\randomset]{\ensuremath{{#1}^\circ}\xspace} 
\newcommand{\closedset}[1][\randomset]{\ensuremath{\bar{#1}}\xspace} 
\newcommand{\boundaryset}[1][\randomset]{\ensuremath{\partial{#1}}\xspace} 
\newcommand{\powerindex}{\ensuremath{\alpha}\xspace} 
\newcommand{\ratefunction}[1][\randomset]{\ensuremath{\mathcal{J}(#1) }\xspace} 
\newcommand{\stepfunctionsymbol}{\ensuremath{\xi}\xspace} 
\newcommand{\stepfunction}[1][t]{\ensuremath{\stepfunctionsymbol(#1)}\xspace} 

\newcommand{\discontinuities}[1][\stepfunctionsymbol]{\ensuremath{\mathcal{D}_+(#1)}\xspace} 
\newcommand{\discontinuitiesset}[1][\stepfunctionsymbol]{\ensuremath{\mathcal{D}(#1)}\xspace} 

\newcommand{\skorokhodfunctionsymbol}{\ensuremath{\zeta}\xspace} 
\newcommand{\homeomorphismsymbol}{\ensuremath{h}\xspace} 
\newcommand{\identitymapping}{\ensuremath{id}\xspace} 

\newcommand{\levydrift}{\ensuremath{c}\xspace} 
\newcommand{\crosslevel}{\ensuremath{a}\xspace} 

\newcommand{\lcrhittingsetsymbol}{\ensuremath{A}\xspace} 
\newcommand{\lcrhittingset}[1][\threshold]{\ensuremath{\lcrhittingsetsymbol^{#1}_{\levydrift,\crosslevel}}\xspace} 
\newcommand{\ecomorhittingsetsymbol}{\ensuremath{\mathcal{A}}\xspace} 
\newcommand{\ecomorhittingset}[1][\threshold]{\ensuremath{\ecomorhittingsetsymbol^{#1}_{\levydrift,\crosslevel}}\xspace} 
\newcommand{\preconstant}[1][\randomset]{\ensuremath{ C_{\threshold+1} (#1) }\xspace} 
\newcommand{\fixedpreconstant}{\ensuremath{ \mathcal{C}_{\threshold+1}}\xspace} 
\newcommand{\LCRpreconstant}{\ensuremath{ \mathcal{C}_{\threshold+1}^{L} }\xspace} 
\newcommand{\ECOMORpreconstant}{\ensuremath{ \mathcal{C}_{\threshold+1}^{E} }\xspace} 
\newcommand{\lowerconstant}[1][\randomset]{\ensuremath{ C_{\ratefunction[#1]} (\openset[#1]) }\xspace} 
\newcommand{\upperconstant}[1][\randomset]{\ensuremath{ C_{\ratefunction[#1]} (\closedset[#1]) }\xspace} 
\newcommand{\limitconstant}[1][\randomset]{\ensuremath{ C_{\ratefunction[#1]} (#1) }\xspace} 
\newcommand{\measureconstant}[1][j]{\ensuremath{ C_{#1} }\xspace} 

\newcommand{\threshold}{\ensuremath{r}\xspace} 
\newcommand{\LCR}[2][t]{\ensuremath{L_{#2}(#1)}\xspace} 
\newcommand{\ECOMOR}[2][t]{\ensuremath{E_{#2}(#1)}\xspace} 
\newcommand{\reinsurance}[1][t]{\ensuremath{R(#1)}\xspace} 
\newcommand{\mappingsymbol}{\ensuremath{\phi}\xspace} 
\newcommand{\lcrmappingsymbol}{\ensuremath{\phi_\threshold}\xspace} 
\newcommand{\ecomormappingsymbol}{\ensuremath{\varphi_\threshold}\xspace} 
\newcommand{\lcrmapping}[1][\stepfunctionsymbol]{\ensuremath{\lcrmappingsymbol(#1)}\xspace} 
\newcommand{\ecomormapping}[1][\stepfunctionsymbol]{\ensuremath{\ecomormappingsymbol(#1)}\xspace} 
\newcommand{\lcrlipschitz}{\ensuremath{K}\xspace} 
\newcommand{\ecomorlipschitz}{\ensuremath{L}\xspace} 

\newcommand{\summaxjump}[3][t]{\ensuremath{{\mathfrak J}^{#2}_{#3}(#1)}\xspace} 
\newcommand{\maxjumpsymbol}{\ensuremath{x} \xspace} 
\newcommand{\maxjump}[1][i]{\ensuremath{\maxjumpsymbol_{#1}} \xspace} 
\newcommand{\maxjumptimesymbol}{\ensuremath{u} \xspace} 
\newcommand{\maxjumptime}[1][i]{\ensuremath{\maxjumptimesymbol_{#1}} \xspace} 

\newcommand{\recursive}[1][n]{\ensuremath{\mathcal{I}_{#1}} \xspace} 

\newcommand{\levyprocess}[1][t]{\ensuremath{Y(#1)}\xspace} 
\newcommand{\levymeasure}{\ensuremath{\nu}\xspace} 
\newcommand{\measurepowerexponent}[2][\powerindex]{\ensuremath{\levymeasure^{#2}_{#1}}\xspace} 
\newcommand{\slowlyvarying}[1][x]{\ensuremath{L(#1)}\xspace} 
\newcommand{\scaledlevyprocess}[2][t]{\ensuremath{\bar{Y}_{#2}(#1)}\xspace} 
\newcommand{\discretescaledlevyprocess}[1][n]{\ensuremath{\bar{Y}_{#1}}\xspace} 

\newcommand{\realdecreasing}[1][j]{\ensuremath{ \reals[#1 \downarrow]_+ }\xspace} 
\newcommand{\coupledset}[1][\threshold+1]{\ensuremath{ S_{#1}}\xspace} 
\newcommand{\exactsteps}[1][\threshold+1]{\ensuremath{ \skorokhodspace_{#1}}\xspace} 
\newcommand{\atmoststeps}[1][\threshold]{\ensuremath{ \skorokhodspace_{ \leqslant #1}}\xspace} 
\newcommand{\stepmappingsymbol}[1][j]{\ensuremath{T_{#1}}\xspace} 
\newcommand{\stepmapping}[2][j]{\ensuremath{\stepmappingsymbol[#1](#2)}\xspace} 
\newcommand{\inversestepmappingsymbol}[1][j]{\ensuremath{T^{-1}_{#1}}\xspace} 
\newcommand{\inversestepmapping}[2][j]{\ensuremath{\inversestepmappingsymbol[#1]\left(#2\right)}\xspace} 

\newcommand{\insurersafetyloading}{\ensuremath{\theta}\xspace} 
\newcommand{\reinsurersafetyloading}{\ensuremath{\eta}\xspace} 

\dottedcontents{section}[3em]{}{2em}{1pc}
\dottedcontents{subsection}[5em]{}{2.4em}{1pc}

\newcommand{\footnoteremember}[2]
{
   \newcounter{#1}\footnote{#2}\setcounter{#1}{\value{footnote}}
}
\newcommand{\footnoterecall}[1]
{
   \footnotemark[\value{#1}]
}

\graphicspath{{graphs/}{logos/}}

\title{Finite-time ruin probabilities under large-claim reinsurance treaties for heavy-tailed claim sizes}
\author{
    Hansj\"{o}rg Albrecher\footnoteremember{UNILSFI}{Department of Actuarial Science, Faculty of Business and Economics and Swiss Finance Institute, University of Lausanne, 1015 Lausanne, Switzerland}\\
    \small \texttt{hansjoerg.albrecher@unil.ch}\\
    \and
    Bohan Chen\footnoteremember{CWI}{Centrum Wiskunde \& Informatica (CWI), P.O. Box 94079, 1090 GB Amsterdam, The Netherlands}\\
    \small \texttt{B.Chen@cwi.nl}\\
    \and
    Eleni Vatamidou\footnoteremember{UNIL}{Department of Actuarial Science, Faculty of Business and Economics, University of Lausanne, 1015 Lausanne, Switzerland}\\
    \small \texttt{eleni.vatamidou@unil.ch}\\
    \and
    Bert Zwart\footnoterecall{CWI}\\
    \small \texttt{Bert.Zwart@cwi.nl}\\
}
\date{ }

\begin{document}

\maketitle

\begin{abstract}
We investigate the probability that an insurance portfolio gets ruined within a finite time period under the assumption that the $r$ largest claims are (partly) reinsured. We show that for regularly varying claim sizes the probability of ruin after reinsurance is also regularly varying in terms of the initial capital, and derive an explicit asymptotic expression for the latter. We establish this result by leveraging recent developments on sample\-/path large deviations for heavy tails. Our results allow, on the asymptotic level, for an explicit comparison between two well\-/known large\-/claim reinsurance contracts, namely LCR and ECOMOR. We finally assess the accuracy of the resulting approximations using state\-/of\-/the\-/art rare event simulation techniques.

\end{abstract}


\section{Introduction}\label{Section: Introduction}
We consider the following ruin problem of the classical Cram{\'e}r\-/Lundberg model in risk theory; see e.g.\ \cite{asmussen-RP}. Let $\{\claim[1],\claim[2],\dots\}$ be a sequence of i.i.d.\ positive random variables representing successive claim sizes that arrive according to a homogeneous Poisson process \poissonprocess, $t \geq 0$, with rate \poissonrate. Premiums are received continuously at a constant rate $\premium > \poissonrate \e \claim[]$. We assume that there is also a reinsurance agreement in place, where \reinsurance is the reinsured amount at time $t$. More precisely, if $\aggregateclaims = \sum_{i=1}^{\poissonprocess} \claim$ is the aggregate claim amount at time $t$ and \cedentpremium is the remaining premium for the insurer after reinsurance has been purchased, then the aggregate loss minus premiums at time $t$ for the insurer is equal to $\aggregateclaims - \cedentpremium t - \reinsurance$. If $\initialcapital \geq 0$ is the initial capital, then the probability of ruin before time $T$ is defined as
\begin{equation}\label{Eq. Definition of the ruin probability in the general reinsurance framework}
    \finiteruinprobability = \pr \Bigg( \sup_{0 \leq t \leq T} \{ \aggregateclaims - \cedentpremium t - \reinsurance \} > \initialcapital \Bigg).
\end{equation}

We will restrict our attention to two forms of large claims reinsurance, namely LCR and ECOMOR. In an LCR (largest claim reinsurance) contract (see e.g.\ \cite{ammeter1964rating} for an early reference), the reinsurer agrees to cover the largest \threshold claims, where $\threshold \geq 1$ is a fixed number, while in an ECOMOR (exc\'{e}dent du co\^{u}t moyen relatif) contract \cite{thepaut1950nouvelle}, the reinsurer covers the excess of the \threshold largest claims over the $(\threshold+1)$st largest claim; see \cite{albrecher-RASA} for more details on this type of reinsurance contracts.

We assume that the distribution of the claim sizes belongs to a class of distributions with a regularly varying tail, which is valid for many applications \cite{embrechts-MEE}. It is well known that  the {\it principle of one big jump} holds in the heavy\-/tailed claim setting, i.e.\ ruin is typically caused by a single large claim. However, under the presence of large claim reinsurance contracts, ruin probabilities are typically harder to analyse because the largest claims are covered by the reinsurer and thus multiple claims may be responsible for the event of ruin.

Several papers have studied properties of large claim reinsurance contracts. For example, when claim sizes are light\-/tailed, the asymptotic tail behavior of the reinsured amounts is considered in \cite{hashorva2013ecomor,jiang2008reinsurance} and their joint tail behavior in \cite{peng2014joint}. For asymptotic properties of the reinsured amounts when the claim size distribution is heavy\-/tailed, see \cite{albrecher2014joint,ladoucette2006reinsurance}. For dependence between claim sizes and interarrival times in this context, see \cite{li2015asymptotics}. An interesting recent link between large claim treaties and risk measures is given in \cite{castanomartinez2019family}. However, none of these contributions deal with the ruin probability, which is considered here.

In this paper, we suggest to leverage recent new tools developed in the context of sample\-/path large deviations for heavy\-/tailed stochastic processes for the study of ruin problems under LCR and ECOMOR treaties. Concretely, for a centered L\'{e}vy process \levyprocess, $t \geq 0$, with regularly varying L\'{e}vy measure \levymeasure, sample\-/path large deviations were developed in \cite{rhee2017sample}. Consider the process $\discretescaledlevyprocess = \{ \scaledlevyprocess{n}, t \in [0,1] \}$, where $\scaledlevyprocess{n} = \levyprocess[n t]/n$, $t \geq 0$. Then, asymptotic estimates of $\pr ( \discretescaledlevyprocess \in \randomset )$ for a large collection of sets \randomset were derived. For L\'{e}vy processes with only positive jumps that are regularly varying with index $-\powerindex$, $\powerindex >1$, these results take the form
\begin{equation}\label{Eq. Sample path large deviations theorem}
    \lowerconstant
    \leq \liminf_{n \to \infty} \frac{\pr ( \discretescaledlevyprocess \in \randomset )}{ \big( n\cdot \levymeasure[n,\infty) \big)^{\ratefunction}}
    \leq \limsup_{n \to \infty} \frac{\pr ( \discretescaledlevyprocess \in \randomset )}{ \big( n\cdot \levymeasure[n,\infty) \big)^{\ratefunction}}
    \leq \upperconstant,
\end{equation}
where \openset and \closedset are the interior and closure of \randomset, \ratefunction is interpreted as the minimum number of jumps in the L\'{e}vy process that are needed to cause the event \randomset, and \measureconstant is a measure. We will show how the reinsurance problem fits in the above framework. For this, we resolve several technical challenges such as showing how ruin probabilities in the reinsurance setting can be written as continuous maps of the input process in a suitable Skorokhod space.

Apart from the fact that reinsurance contracts are an interesting object of study in their own right, the present application seems to be the first example for which it is possible to compute the pre\-/factors in the asymptotics \eqref{Eq. Sample path large deviations theorem} explicitly. More precisely, we show for both the LCR and ECOMOR treaty that $\lowerconstant = \upperconstant$ and we provide an explicit expression for this value.

The rest of the paper is organised as follows. In \Cref{Section: Model description and preliminaries}, we provide some preliminary results and introduce the necessary notation. \Cref{Section: Main result} develops the main result, i.e.\ the tail asymptotics for finite time ruin probabilities. For this, we are inquired to write \eqref{Eq. Definition of the ruin probability in the general reinsurance framework} in terms of \eqref{Eq. Sample path large deviations theorem}. This leads to the need to show continuity of certain mappings, as well as several additional technical requirements. In \Cref{Section: Numerical implementations}, we validate our asymptotic results with numerical experiments.

\section{Model description and preliminaries}\label{Section: Model description and preliminaries}

Following the notation and terminology used in \Cref{Section: Introduction}, let \claimsizedissymbol denote the distribution function of the claim sizes and $\e \claim[]$ be their expectation. We assume that \claimsizedissymbol is regularly varying with index $-\powerindex$, i.e.\ there exists a slowly varying function \slowlyvarying such that $\comclaimsizedis : = 1- \claimsizedis = \slowlyvarying x^{-\powerindex}$, with $\powerindex >1$.
Let further $\orderedclaim[1]{t} \geq \orderedclaim[2]{t} \geq \cdots \geq \orderedclaim[\poissonprocess]{t}$ denote the order statistics of $\claim[1],\claim[2],\dots\claim[\poissonprocess]$.

In an LCR treaty, the reinsured amount \reinsurance is equal to
\begin{equation}\label{Eq. LCR reinsurance}
    \LCR{\threshold} := \sum_{i=1}^\threshold \orderedclaim{t},
\end{equation}
i.e.\ the \threshold largest claims are paid by the reinsurer.
On the other hand, the reinsured amount \reinsurance in an ECOMOR treaty takes the form
\begin{equation}\label{Eq. ECOMOR reinsurance}
    \ECOMOR{\threshold}:= \sum_{i=1}^\threshold \orderedclaim{t} - \threshold \orderedclaim[\threshold+1]{t}
    = \sum_{i=1}^{\poissonprocess} \big( \claim - \orderedclaim[\threshold+1]{t} \big)_+.
\end{equation}

That is, the ECOMOR constitutes an excess\-/of\-/loss treaty with a random retention, and the latter is the $(\threshold+1)$-largest claim. For more details and background on such reinsurance contracts, see \cite{albrecher-RASA}.
In either treaty, the number of reinsured claims is equal to \threshold.

\begin{assumption}\label{Assumption: Zero retention level}
If $\poissonprocess \leq \threshold$, we set $\orderedclaim{t}=0$, for $i= \poissonprocess+1,\dots,\threshold +1$. This means that in case there are less than $\threshold+1$ claims, the reinsurer pays all the claims in the ECOMOR treaty.
\end{assumption}

Another modeling assumption is concerned with the way the reinsurance is affecting the capital position of the insurance company under consideration.

\begin{assumption}\label{Assumption: Immediate payment of claims}
We assume that at each time $t$, the currently applicable reinsured amount \reinsurance is considered in the determination of the available surplus. In particular, this means that before the arrival of the $(\threshold+1)$-st claim, the random retention in the ECOMOR treaty is considered to be zero. As a consequence in the ECOMOR treaty, the arrival of a new claim can lead to a modification of \reinsurance of either sign, as the excess over the $(\threshold+1)$-st claim may also decrease.
\end{assumption}

Note also that the setup we study here is that the duration of the reinsurance contract is $T$, and the implied premium for the reinsurance contract over the period $[0,T]$ is uniformly spread over this time interval. We will study the asymptotic behavior of the finite time ruin probabilities \eqref{Eq. Definition of the ruin probability in the general reinsurance framework} utilizing \eqref{Eq. Sample path large deviations theorem}. Therefore, we formulate in the next section the large deviation problem that arises in our reinsurance context.


\subsection{Large deviations in reinsurance}\label{Section: Large deviations in reinsurance}
In \cite{rhee2017sample}, the large deviations results \eqref{Eq. Sample path large deviations theorem} were derived in the Skorokhod \skorokhodtopology topology. Correspondingly, we let $\skorokhodspace = \skorokhodspace ([0,1],\reals)$ be a Skorokhod space, i.e.\ a space of real\-/valued c\`{a}dl\`{a}g (right continuous with left limits) functions on $[0,1]$, equipped with the \skorokhodtopology-metric defined by
\begin{equation}\label{Eq. Definition Skorokhod metric}
    \distance[\stepfunctionsymbol,\skorokhodfunctionsymbol] = \inf_{\homeomorphismsymbol \in \bijectionset} \{ \supnorm[\homeomorphismsymbol - \identitymapping ] \vee \supnorm[\stepfunctionsymbol - \skorokhodfunctionsymbol \circ \homeomorphismsymbol] \}, \qquad (\stepfunctionsymbol,\skorokhodfunctionsymbol) \in \skorokhodspace[2],
\end{equation}
where \bijectionset denotes the set of all strictly increasing continuous bijections from $[0,1]$ to itself, \identitymapping denotes the identity mapping, and \supnorm denotes the uniform (sup) norm on $[0,1]$. Thus, \randomset and \measureconstant in \eqref{Eq. Sample path large deviations theorem} are a measurable set and a measure on \skorokhodspace, respectively. Furthermore, if $\mappingsymbol : \skorokhodspace \to \reals$ is a continuous functional on \skorokhodspace and $\borelset \in \realborel$ is a Borel set such that $\randomset = \mappingsymbol^{-1}(\borelset)$, where $\mappingsymbol^{-1}$ stands for the inverse of \mappingsymbol, it holds that
\begin{equation}\label{Eq. Connection between the probability of the levy and its continuous functional}
    \pr \big( \mappingsymbol (\discretescaledlevyprocess) \in \borelset \big)
    =\pr \big( \discretescaledlevyprocess \in \mappingsymbol^{-1}(\borelset) \big)
    =\pr ( \discretescaledlevyprocess \in \randomset ).
\end{equation}

The above relation portrays how it is possible to use the result \eqref{Eq. Sample path large deviations theorem} to study continuous functionals of \discretescaledlevyprocess. To connect this to our ruin problem, we define $ \scaledaggregateclaimssequence:= \{ \scaledaggregateclaims{n}, t \in [0,1] \}$ as the centred and scaled process
\begin{equation}\label{Eq. Definition of the scaled and centred process}
    \scaledaggregateclaims{n} = \frac1n \aggregateclaims[nt] - \poissonrate \e \claim[] t=
    \frac1n \sum_{i=1}^{\poissonprocess[nt]} \claim - \poissonrate \e \claim[] t, \quad t \geq 0.
\end{equation}
Moreover, we assume that the capital \initialcapital increases linearly in $n$, i.e.\ there exists an $\crosslevel>0$ such that $\initialcapital = n \crosslevel$. We now formulate the large deviations problem to estimate the probabilities
\begin{align}
    &\pr \Bigg( \sup_{t\in[0,1]} \{ \aggregateclaims[nt] - \cedentpremium nt - \reinsurance[nt] \} \geq n \crosslevel \Bigg) \notag \\
    &= \pr \Bigg( \sup_{t\in[0,1]} \{ \aggregateclaims[nt] - \poissonrate \e \claim[] nt- (\cedentpremium - \poissonrate \e \claim[]) nt - \reinsurance[nt] \} \geq n \crosslevel \Bigg) \notag \\
    &=  \pr \Bigg( \sup_{t\in[0,1]} \{ n \scaledaggregateclaims{n} - \levydrift nt -  \reinsurance[nt] \} \geq n\crosslevel \Bigg)
    =  \pr \Bigg( \sup_{t\in[0,1]} \{ \scaledaggregateclaims{n} - \levydrift t -  \frac1n \reinsurance[nt] \} \geq \crosslevel \Bigg), \label{Eq. Large deviations problem}
\end{align}
where $\levydrift = \cedentpremium -  \poissonrate \e \claim[]$\label{parameter c}.
As a next step, we must identify a continuous functional \mappingsymbol such that
\begin{equation}\label{Eq. The continuous functional with the general reinsurance}
    \sup_{t\in[0,1]} \{ \scaledaggregateclaims{n} - \levydrift t -  \frac1n \reinsurance[nt] \}
    = \mappingsymbol (\scaledaggregateclaimssequence),
\end{equation}
so that we can write
\begin{equation}\label{Eq. Probability of the scaled reinsurance}
\pr \Bigg( \sup_{t\in[0,1]} \{ \scaledaggregateclaims{n} - \levydrift t -  \frac1n \reinsurance[nt] \} \geq \crosslevel \Bigg)
= \pr \big( \mappingsymbol (\scaledaggregateclaimssequence) \geq \crosslevel \big)
= \pr \Big( \scaledaggregateclaimssequence \in \mappingsymbol^{-1}\big( [\crosslevel, \infty) \big) \Big).
\end{equation}

However, it is not immediately obvious from \Cref{Eq. The continuous functional with the general reinsurance} what the functional \mappingsymbol looks like because \reinsurance[nt] is not expressed in terms of \scaledaggregateclaimssequence. We focus first on the LCR treaty and observe that
\begin{equation*}
    \frac1n \reinsurance[nt]
    =\frac1n \LCR[nt]{\threshold}
    =\frac1n \sum_{i=1}^\threshold \orderedclaim{nt}
    =\max_{\substack{ (s_1,\dots,s_\threshold)\in[0,t]^\threshold\\s_i\neq s_j,\forall i\neq j}} \sum_{i=1}^\threshold \big( \scaledaggregateclaims[s_i]{n}-\scaledaggregateclaims[s_i^-]{n} \big), \quad t \in [0,1],
\end{equation*}
i.e.\ $\LCR[nt]{\threshold}/n$ can be expressed as the sum of the \threshold biggest jumps of the process \scaledaggregateclaims{n}. For every $\stepfunctionsymbol \in \skorokhodspace$ and $m \in \naturals$, we define
\begin{equation}\label{Eq. Definition of sum of largest claim sizes}
    \summaxjump{m}{\stepfunctionsymbol}
    =
    \sup_{\substack{ (s_1,\dots,s_m)\in[0,t]^m\\s_i\neq s_j,\forall i\neq j}} \sum_{i=1}^m \big( \stepfunction[s_i]-\stepfunction[s_i^-] \big) =
    \max_{\substack{ (s_1,\dots,s_m)\in[0,t]^m\\s_i\neq s_j,\forall i\neq j}} \sum_{i=1}^m \big( \stepfunction[s_i]-\stepfunction[s_i^-] \big), \quad \text{for } t\in(0,1],
\end{equation}
as the supremum of the sum of the $m$ largest jumps of the function \stepfunctionsymbol. Naturally, $\summaxjump[0]{m}{\stepfunctionsymbol} =0$.
Consequently, the functional \mappingsymbol we are looking for is a mapping $\lcrmappingsymbol: \skorokhodspace \to \reals$ defined for every $\stepfunctionsymbol \in \skorokhodspace$ as
\begin{equation}\label{Eq. Mapping LCR}
    \lcrmapping = \sup_{t \in [0,1]} \Big\{ \stepfunction  - \levydrift t  - \summaxjump{\threshold}{\stepfunctionsymbol}\Big\}.
\end{equation}
Moreover, we denote the pre\-/image of $[\crosslevel,\infty)$ under \lcrmappingsymbol as $\lcrhittingset = \lcrmappingsymbol^{-1}\big( [\crosslevel,\infty) \big)$ where
\begin{equation}\label{Eq. Hitting set LCR}
    \lcrhittingset
    = \Bigg\{  \stepfunctionsymbol \in \skorokhodspace: \sup_{t \in [0,1]} \Big\{ \stepfunction  - \levydrift t  - \summaxjump{\threshold}{\stepfunctionsymbol}\Big\} \geq \crosslevel \Bigg\}.
\end{equation}

By comparing \Cref{Eq. LCR reinsurance,Eq. ECOMOR reinsurance}, we observe that the relation between the reinsured amounts of the two treaties is
\begin{equation*}
    \ECOMOR{\threshold}
    = \LCR{\threshold} - \threshold \orderedclaim[\threshold+1]{t}
    = (\threshold + 1)\LCR{\threshold} - \threshold \big( \LCR{\threshold} + \orderedclaim[\threshold+1]{t} \big)
    = (\threshold + 1)\LCR{\threshold} - \threshold \LCR{\threshold+1}.
\end{equation*}
Thus, in the ECOMOR treaty, the functional \mappingsymbol in \eqref{Eq. Probability of the scaled reinsurance} is the mapping $\ecomormappingsymbol: \skorokhodspace \to \reals$ defined for every $\stepfunctionsymbol \in \skorokhodspace$ as
\begin{equation}\label{Eq. Mapping ECOMOR}
    \ecomormapping
    = \sup_{t \in [0,1]} \Big\{ \stepfunction  - \levydrift t  - (\threshold + 1)\summaxjump{\threshold}{\stepfunctionsymbol} + \threshold \summaxjump{\threshold+1}{\stepfunctionsymbol} \Big\},
\end{equation}
while the pre\-/image of $[\crosslevel,\infty)$ under \ecomormappingsymbol, i.e.\ $\ecomorhittingset = \ecomormappingsymbol^{-1}\big( [\crosslevel,\infty) \big)$, is defined as
\begin{equation}\label{Eq. Hitting set ECOMOR}
    \ecomorhittingset
    = \Bigg\{  \stepfunctionsymbol \in \skorokhodspace: \sup_{t \in [0,1]} \Big\{ \stepfunction  - \levydrift t  - (\threshold + 1)\summaxjump{\threshold}{\stepfunctionsymbol} + \threshold \summaxjump{\threshold+1}{\stepfunctionsymbol} \Big\} \geq \crosslevel \Bigg\}.
\end{equation}

\subsection{Preliminaries on the Skorokhod topology and notation}\label{Section: Preliminaries on the Skorokhod topology and notation}
Consider the complete metric space $\big(\skorokhodspace,\distance[,]\big)$. The functional \summaxjump{m}{\stepfunctionsymbol} defined in \eqref{Eq. Definition of sum of largest claim sizes} will play a significant role in the forthcoming analysis. Thus, it is important to confirm that it is well\-/defined. For this reason, let \discontinuitiesset be the set of discontinuities of $\stepfunctionsymbol \in \skorokhodspace$, i.e.
\begin{equation}
    \discontinuitiesset = \{  t \in [0,1]: \stepfunction[t^-] \neq \stepfunction\},
\end{equation}
and let \discontinuitiesset[\stepfunctionsymbol,\epsilon] be the set of discontinuities of magnitude at least $\epsilon$, i.e.
\begin{equation}\label{Eq. definition of set of disc. of magnitude at least epsilon}
    \discontinuitiesset[\stepfunctionsymbol,\epsilon] =
    \{  t \in [0,1]: \abs[{\stepfunction[t^-] - \stepfunction}] \geq \epsilon\}.
\end{equation}
The following result is standard.
\begin{lemma}[Theorem~12.2.1 \& Corollary~12.2.1 of \cite{whitt-SPL}]\label{Lemma: Regularity of the Skorokhod space}
For any $\stepfunctionsymbol \in \skorokhodspace$ and $\epsilon>0$,  \discontinuitiesset[\stepfunctionsymbol,\epsilon] is a finite subset of $[0,1]$.
In particular,  \discontinuitiesset is either finite or countably infinite.
\end{lemma}

Consequently, the supremum in \Cref{Eq. Definition of sum of largest claim sizes} is attained because only finitely many jumps can exceed a given positive number. As a result, \summaxjump{m}{\stepfunctionsymbol} is well defined.

Some important subspaces of \skorokhodspace for our analysis are those restricted to step functions. We let \increasingstepset be the set of all non\-/decreasing step functions vanishing at the origin. Furthermore, \exactsteps[j] is the subspace of \skorokhodspace consisting of non\-/decreasing step functions, vanishing at the origin, with exactly $j$ steps, and similarly, $\atmoststeps[j] = \bigcup_{0 \leq i \leq j} \exactsteps[i]$ consists of non\-/decreasing step functions, vanishing at the origin, with at most $j$ steps. Finally, if \discontinuities denotes the number of discontinuities of $\stepfunctionsymbol \in \skorokhodspace$, we can then formally define the integer\-/valued set function \ratefunction appearing in \Cref{Eq. Sample path large deviations theorem} by:
\begin{equation}\label{Eq. Formal definition of the rate function}
    \ratefunction = \inf_{\stepfunctionsymbol \in \randomset \cap \increasingstepset} \discontinuities,
\end{equation}
which we call the rate function. Observe that every $\stepfunctionsymbol \in \exactsteps[j]$ is determined by the pair of jump sizes and jump times $(\vect{\maxjumpsymbol},\vect{\maxjumptimesymbol}) \in \reals[j]_+ \times [0,1]^j$, i.e.\ $\stepfunction =  \sum_{i = 1}^{j} \maxjump \indicatorfunction[\{\maxjumptime,1\}](t)$, where \indicatorfunction[\borelset] is the indicator function on the set \borelset. For $ \vect{\maxjumpsymbol}= (\maxjump[1],\dots,\maxjump[j])$ and $ \vect{\maxjumptimesymbol}= (\maxjumptime[1],\dots,\maxjumptime[j])$, we define the sets
\begin{align}
    \realdecreasing
    &= \{ \vect{\maxjumpsymbol} \in \reals[j]_+ : \maxjump[1] \geq \maxjump[2] \geq \dots  \geq \maxjump[j]>0\},
\intertext{and}
    \coupledset[j]
    &= \{ (\vect{\maxjumpsymbol}, \vect{\maxjumptimesymbol}) \in \realdecreasing \times (0,1)^j:  \maxjumptime[1],\dots,\maxjumptime[j] \text{ are all distinct} \},
\end{align}
where the $\maxjumptime[j]$'s are not following the ordering of the $\maxjump[j]$'s, i.e.\ $\maxjump[k] \geq \maxjump[l] \not \Rightarrow \maxjumptime[k] \geq \maxjumptime[l]$. Thus, we can formally define the mapping $\stepmappingsymbol: \coupledset[j] \to \exactsteps[j]$ by $\stepmapping{\vect{\maxjumpsymbol},\vect{\maxjumptimesymbol}} = \sum_{i = 1}^{j} \maxjump \indicatorfunction[\{\maxjumptime,1\}]$.

Furthermore, let $\measurepowerexponent{}(x,\infty) = x^{-\powerindex}$ (i.e.\ the pure power decay part of the regularly varying claim sizes), and let \measurepowerexponent{j} denote the restriction to \realdecreasing of the $j$\-/fold product measure of \measurepowerexponent{}. We define for each $j \geq 1$ the measure \measureconstant concentrated on \exactsteps[j] as
\begin{equation}\label{Eq. The measure of the constants}
    \measureconstant(\bullet)
    = \e \big[ \measurepowerexponent{j} \{ \vect{y} \in \reals[j]_+ : \sum_{i=1}^j y_i \indicatorfunction[\{\uniform,1\}] \in \bullet \} \big],
\end{equation}
where the random variables \uniform, $i=1,\dots,j$, are i.i.d.\ uniform on $[0,1]$.

Finally, we say that a set $\randomset \subseteq \skorokhodspace$ is bounded away from another set $\borelset \subseteq \skorokhodspace$ if $\inf_{x \in \randomset, y \in \borelset} \distance[x,y]>0$. Additionally, we let $\closeness{\randomset} = \{ \stepfunctionsymbol \in \skorokhodspace: \distance[\stepfunctionsymbol,\randomset] \leq \smallnumber \}$ for any $\smallnumber>0$.

\section{Main result}\label{Section: Main result}
Note that the parameter $\levydrift = \cedentpremium -  \poissonrate \e \claim[]$ introduced in \Cref{Section: Large deviations in reinsurance} can be either positive or negative. However, for $\crosslevel \leq -\levydrift $, the rare event probability in \Cref{Eq. Probability of the scaled reinsurance} converges to one by the functional law of large numbers. For this reason, we focus only on the case $\levydrift+\crosslevel>0$. Letting $\displaystyle \prescript{}{2}{F}_1(b,e;d;z) = \sum_{k=0}^{+\infty} \frac{(b)_k (e)_k}{ (d)_k} \frac{z^k}{k!}$ be the hypergeometric function, with $(b)_k = b (b+1) \dots (b+k-1)$ denoting the Pochhammer symbol, we have the following theorem.

\begin{theorem}\label{Theorem: Asymptotic ruin probability}
For $\crosslevel>0$, $\levydrift+\crosslevel>0$, and $\threshold \in \naturals$, it holds that
\begin{equation}\label{Eq. Expression of main result}
\asymptotfiniteruinprobability
\sim
\fixedpreconstant \big(\poissonrate \slowlyvarying[n] \big)^{\threshold+1} n^{-(\threshold+1)(\powerindex-1)}, \qquad n \to \infty,
\end{equation}
where
\begin{gather*}
    \fixedpreconstant
    = \Bigg[ \crosslevel^{-(\threshold + 1) \powerindex} \prescript{}{2}{F}_1[\threshold+1, (\threshold + 1) \powerindex; \threshold + 2; -\levydrift/\crosslevel] \cdot \indicatorfunction[\{c>0\}]
    + (\crosslevel + \levydrift)^{-(\threshold+1)\powerindex} \cdot \indicatorfunction[\{c<0\}] \Bigg] \times \frac{1}{(\threshold+1)!} \\
    \times
    \begin{cases}
    1, & \text{if } \reinsurance = \LCR{\threshold} \text{ (LCR)}, \\
    (\threshold+1)^{(\threshold+1)\powerindex}, & \text{if } \reinsurance = \ECOMOR{\threshold} \text{ (ECOMOR)}.
    \end{cases}
\end{gather*}
\end{theorem}

The proof of \Cref{Theorem: Asymptotic ruin probability} is based on sample\-/path large\-/deviations results developed in \cite{rhee2017sample}. Specifically, Theorems~3.1--3.2 in \cite{rhee2017sample} provide the conditions under which the result \eqref{Eq. Sample path large deviations theorem} holds, and in addition the $\liminf$ and $\limsup$ are equal. Thus, to achieve our goal, we must verify that these conditions are satisfied for $\discretescaledlevyprocess = \scaledaggregateclaimssequence$ and $\randomset = \lcrhittingset$ (LCR) or $\randomset = \ecomorhittingset$ (ECOMOR) defined in \Cref{Eq. Hitting set LCR,Eq. Hitting set ECOMOR}, respectively. However, their verification is rather involved. Hence, to make the proof of \Cref{Theorem: Asymptotic ruin probability} more accessible, we split it in various steps after the aforementioned conditions and we provide additional explanations for each step.

Note that all of the forthcoming results are similar in the two treaties with possible deviations in small details. Therefore, we will first prove them for the LCR treaty and then show briefly how they can be extended to the ECOMOR treaty.

\subsection{Proof of \Cref*{Theorem: Asymptotic ruin probability}}
The first step is to show that both mappings $\lcrmappingsymbol, \ecomormappingsymbol: \skorokhodspace \to \reals$ from \Cref{Eq. Mapping LCR,Eq. Mapping ECOMOR}, respectively, are Lipschitz continuous. Due to their continuity, \Cref{Eq. Probability of the scaled reinsurance} will hold and, consequently, we will be able to write $\pr \big( \lcrmappingsymbol (\scaledaggregateclaimssequence) \geq \crosslevel \big)
= \pr ( \scaledaggregateclaimssequence \in \lcrhittingset )$ and $\pr \big( \ecomormappingsymbol (\scaledaggregateclaimssequence) \geq \crosslevel \big)
= \allowbreak \pr ( \scaledaggregateclaimssequence \in \ecomorhittingset )$. For this, we need the following intermediate result.
\begin{lemma}\label{Lemma. Useful inequality about the sum of the max jump}
For every $(\stepfunctionsymbol, \skorokhodfunctionsymbol) \in \skorokhodspace[2]$, $m \in \naturals$, and $\homeomorphismsymbol \in \bijectionset$, it holds that
\begin{equation}
    \abs[{\summaxjump{m}{\skorokhodfunctionsymbol \circ \homeomorphismsymbol}  - \summaxjump{m}{\stepfunctionsymbol} }] \leq 2 m \supnorm[\stepfunctionsymbol - \skorokhodfunctionsymbol \circ \homeomorphismsymbol], \qquad \forall t \in [0,1].
\end{equation}
\end{lemma}
\begin{proof}
By the definition of \summaxjump{m}{\skorokhodfunctionsymbol \circ \homeomorphismsymbol}, there exists $(\sigma_1,\dots,\sigma_m) \in [0,t]^m$ with $\sigma_i \neq \sigma_j$ for all $i \neq j$, such that
\begin{equation}\label{Eq. Estimation sum of first r largest jumps}
    \summaxjump{m}{\skorokhodfunctionsymbol \circ \homeomorphismsymbol}
    = \sum_{i=1}^m \big( \skorokhodfunctionsymbol \circ \homeomorphismsymbol(\sigma_i)-\skorokhodfunctionsymbol \circ \homeomorphismsymbol(\sigma_i^-) \big).
\end{equation}
In addition, we have that
\begin{equation}\label{Eq. Inequality sum of first r largest jumps}
    \summaxjump{m}{\stepfunctionsymbol}
    = \max_{\substack{ (s_1,\dots,s_m)\in[0,t]^m\\s_i\neq s_j,\forall i\neq j}} \sum_{i=1}^m \big( \stepfunction[s_i]-\stepfunction[s_i^-] \big)
    \geq \sum_{i=1}^m \big( \stepfunction[\sigma_i]-\stepfunction[\sigma_i^-] \big).
\end{equation}
Subtracting now \Cref{Eq. Estimation sum of first r largest jumps,Eq. Inequality sum of first r largest jumps}, we obtain
\begin{align*}
    \summaxjump{m}{\skorokhodfunctionsymbol \circ \homeomorphismsymbol}
    - \summaxjump{m}{\stepfunctionsymbol}
    &\leq \sum_{i=1}^m \big( \skorokhodfunctionsymbol \circ \homeomorphismsymbol(\sigma_i)-\skorokhodfunctionsymbol \circ \homeomorphismsymbol(\sigma_i^-)  - \stepfunction[\sigma_i]+\stepfunction[\sigma_i^-] \big)\\
    &\leq \sum_{i=1}^m \big( \abs[{\skorokhodfunctionsymbol \circ \homeomorphismsymbol(\sigma_i) - \stepfunction[\sigma_i]}] + \abs[{\skorokhodfunctionsymbol \circ \homeomorphismsymbol(\sigma_i^-) - \stepfunction[\sigma_i^-]}] \big)\\
    &\leq 2 m \supnorm[\stepfunctionsymbol - \skorokhodfunctionsymbol \circ \homeomorphismsymbol].
\end{align*}
Following similar arguments, we can also show that $\summaxjump{m}{\stepfunctionsymbol} - \summaxjump{m}{\skorokhodfunctionsymbol \circ \homeomorphismsymbol} \leq 2 m \supnorm[\stepfunctionsymbol - \skorokhodfunctionsymbol \circ \homeomorphismsymbol]$, which completes the proof.
\end{proof}

We are now ready to establish the desired continuity.

\begin{lemma}[Lipschitz continuity of the mapping]
The mappings $\lcrmappingsymbol, \ecomormappingsymbol: \skorokhodspace \to \reals$ defined by \Cref{Eq. Mapping LCR,Eq. Mapping ECOMOR}, respectively, are Lipschitz continuous w.r.t.\ \skorokhodtopology. More precisely, there exist $\lcrlipschitz \in [0,\abs[\levydrift] + 2 \threshold + 1]$ and $\ecomorlipschitz \in [0,\abs[\levydrift] + 4 \threshold^2 + 4 \threshold + 1]$ such that $\abs[{\lcrmapping - \lcrmapping[\skorokhodfunctionsymbol]}] \leq \lcrlipschitz \distance[\stepfunctionsymbol,\skorokhodfunctionsymbol]$ and $\abs[{\ecomormapping - \ecomormapping[\skorokhodfunctionsymbol]}] \leq \ecomorlipschitz \distance[\stepfunctionsymbol,\skorokhodfunctionsymbol]$, for all $(\stepfunctionsymbol, \skorokhodfunctionsymbol) \in \skorokhodspace[2]$.
\end{lemma}
\begin{proof}
W.l.o.g.\ we assume that $\lcrmapping \geq \lcrmapping[\skorokhodfunctionsymbol]$, otherwise we switch the roles of \stepfunctionsymbol and \skorokhodfunctionsymbol. For every $\epsilon>0$, there exists $t_* \in [0,1]$ such that
\begin{equation}\label{Eq. LCR continuity inequality -1}
    \stepfunction[t_*] - \levydrift t_*  - \summaxjump[t_*]{\threshold}{\stepfunctionsymbol} > \lcrmapping - \epsilon.
\end{equation}
On the other hand, by the definition of \skorokhodtopology, there exists $\homeomorphismsymbol = \homeomorphismsymbol(\stepfunctionsymbol,\skorokhodfunctionsymbol,\epsilon) \in \bijectionset$ so that
\begin{equation}\label{Eq. Choice homeomorphism}
    \distance[\stepfunctionsymbol,\skorokhodfunctionsymbol] + \epsilon
    =  \supnorm[\homeomorphismsymbol - \identitymapping ] \vee \supnorm[\stepfunctionsymbol - \skorokhodfunctionsymbol \circ \homeomorphismsymbol]
    \geq \big(\homeomorphismsymbol(t_*)  - t_* \big) \vee \big( \stepfunction[t_*] - \skorokhodfunctionsymbol \circ \homeomorphismsymbol(t_*) \big).
\end{equation}
Furthermore, using the fact that \homeomorphismsymbol is a homeomorphism on $[0,1]$, we obtain
\begin{align}
    &\skorokhodfunctionsymbol \circ \homeomorphismsymbol(t_*) - \levydrift \homeomorphismsymbol(t_*) - \summaxjump[t_*]{\threshold}{\skorokhodfunctionsymbol \circ \homeomorphismsymbol} \notag \\
    &= \skorokhodfunctionsymbol \circ \homeomorphismsymbol(t_*) - \levydrift \homeomorphismsymbol(t_*) - \max_{\substack{ (s_1,\dots,s_\threshold)\in[0,t_*]^\threshold\\s_i\neq s_j,\forall i\neq j}} \sum_{i=1}^\threshold \big( \skorokhodfunctionsymbol \circ \homeomorphismsymbol(s_i)- \skorokhodfunctionsymbol \circ \homeomorphismsymbol(s_i^-) \big) \notag \\
    &= \skorokhodfunctionsymbol \big( \homeomorphismsymbol(t_*)\big) - \levydrift \homeomorphismsymbol(t_*) - \max_{\substack{ (s_1,\dots,s_\threshold)\in[0,\homeomorphismsymbol(t_*)]^\threshold\\s_i\neq s_j,\forall i\neq j}} \sum_{i=1}^\threshold \big( \skorokhodfunctionsymbol(s_i)- \skorokhodfunctionsymbol(s_i^-) \big) \notag \\
    &= \skorokhodfunctionsymbol \big(  \homeomorphismsymbol(t_*) \big) - \levydrift \homeomorphismsymbol(t_*) - \summaxjump[\homeomorphismsymbol(t_*)]{\threshold}{\skorokhodfunctionsymbol} \leq \lcrmapping[\skorokhodfunctionsymbol].\label{Eq. LCR continuity inequality -2}
\end{align}
Subtracting \eqref{Eq. LCR continuity inequality -2} from \eqref{Eq. LCR continuity inequality -1} yields
\begin{align*}
    \lcrmapping -   \lcrmapping[\skorokhodfunctionsymbol]
    &< \epsilon + \big( \stepfunction[t_*] - \skorokhodfunctionsymbol \circ \homeomorphismsymbol(t_*) \big) + \levydrift \big(\homeomorphismsymbol(t_*)  - t_* \big) + \big( \summaxjump[t_*]{\threshold}{\skorokhodfunctionsymbol \circ \homeomorphismsymbol}
    - \summaxjump[t_*]{\threshold}{\stepfunctionsymbol} \big) \notag \\
    &< \epsilon + \big( \distance[\stepfunctionsymbol,\skorokhodfunctionsymbol] + \epsilon \big) + \abs[\levydrift] \big( \distance[\stepfunctionsymbol,\skorokhodfunctionsymbol] + \epsilon \big) + 2 \threshold \big( \distance[\stepfunctionsymbol,\skorokhodfunctionsymbol] + \epsilon \big) \notag   \\
    &= (2 + \abs[\levydrift] + 2 \threshold) \epsilon + (1 + \abs[\levydrift] + 2 \threshold) \distance[\stepfunctionsymbol,\skorokhodfunctionsymbol],
\end{align*}
where we have also used \eqref{Eq. Choice homeomorphism} and $\summaxjump[t_*]{\threshold}{\skorokhodfunctionsymbol \circ \homeomorphismsymbol}
- \summaxjump[t_*]{\threshold}{\stepfunctionsymbol} \leq 2 \threshold \supnorm[\stepfunctionsymbol - \skorokhodfunctionsymbol \circ \homeomorphismsymbol]$ by applying \Cref{Lemma. Useful inequality about the sum of the max jump} with $t=t_*$ and $m=\threshold$. Letting $\epsilon \to 0$, we conclude that $\lcrmapping - \lcrmapping[\skorokhodfunctionsymbol] \leq (1 + \abs[\levydrift] + 2 \threshold) \distance[\stepfunctionsymbol,\skorokhodfunctionsymbol]$, i.e.\ \lcrmappingsymbol is Lipschitz continuous. The Lipschitz continuity for the \ecomormappingsymbol mapping can be shown in an analogous manner. More precisely, for every $\epsilon>0$, there exists $t_* \in [0,1]$ such that
\begin{equation}\label{Eq. ECOMOR continuity inequality -1}
    \stepfunction[t_*] - \levydrift t_*  - (\threshold+1) \summaxjump[t_*]{\threshold}{\stepfunctionsymbol} + \threshold \summaxjump[t_*]{\threshold+1}{\stepfunctionsymbol} > \ecomormapping - \epsilon.
\end{equation}
For a homeomorphism \homeomorphismsymbol on $[0,1]$ satisfying \Cref{Eq. Choice homeomorphism}, we have
\begin{align}
    &\skorokhodfunctionsymbol \circ \homeomorphismsymbol(t_*) - \levydrift \homeomorphismsymbol(t_*) - (\threshold+1) \summaxjump[t_*]{\threshold}{\skorokhodfunctionsymbol \circ \homeomorphismsymbol} + \threshold \summaxjump[t_*]{\threshold+1}{\skorokhodfunctionsymbol \circ \homeomorphismsymbol} \notag \\
    &= \skorokhodfunctionsymbol \big( \homeomorphismsymbol(t_*) \big) - \levydrift \homeomorphismsymbol(t_*)
    - (\threshold+1) \summaxjump[\homeomorphismsymbol(t_*)]{\threshold}{\skorokhodfunctionsymbol} + \threshold \summaxjump[\homeomorphismsymbol(t_*)]{\threshold+1}{\skorokhodfunctionsymbol}
    \leq \ecomormapping[\skorokhodfunctionsymbol].\label{Eq. ECOMOR continuity inequality -2}
\end{align}
We assume now w.l.o.g.\ that $\ecomormapping \geq \ecomormapping[\skorokhodfunctionsymbol]$ and we subtract \eqref{Eq. ECOMOR continuity inequality -2} from \eqref{Eq. ECOMOR continuity inequality -1} to attain
\begin{align*}
    \ecomormapping -   \ecomormapping[\skorokhodfunctionsymbol]
    <& \ \epsilon + \big( \stepfunction[t_*] - \skorokhodfunctionsymbol \circ \homeomorphismsymbol(t_*) \big) + \levydrift \big(\homeomorphismsymbol(t_*)  - t_* \big) + (\threshold +1) \big( \summaxjump[t_*]{\threshold}{\skorokhodfunctionsymbol \circ \homeomorphismsymbol}
    - \summaxjump[t_*]{\threshold}{\stepfunctionsymbol} \big) \\
    &+ \threshold  \big( \summaxjump[t_*]{\threshold+1}{\stepfunctionsymbol} - \summaxjump[t_*]{\threshold+1}{\skorokhodfunctionsymbol \circ \homeomorphismsymbol}\big) \notag \\
    <& \ \ \big( 2 + \abs[\levydrift] + 4 \threshold (\threshold +1) \big) \epsilon + \big( 1 + \abs[\levydrift] +  4 \threshold (\threshold +1) \big) \distance[\stepfunctionsymbol,\skorokhodfunctionsymbol],
\end{align*}
where we have also used \eqref{Eq. Choice homeomorphism} and twice \Cref{Lemma. Useful inequality about the sum of the max jump} with $m=\threshold,\threshold+1$ and $t=t_*$. Letting $\epsilon \to 0$, the result is immediate.
\end{proof}

As a next step, we calculate the rate functions \ratefunction[\lcrhittingset] and \ratefunction[\ecomorhittingset] that appear in \Cref{Eq. Sample path large deviations theorem} and are formally defined in \eqref{Eq. Formal definition of the rate function}. We set for simplicity $\levydrift_+=\max \{0,\levydrift\}$ and $\levydrift_-=\max \{0,-\levydrift\}$.

\begin{lemma}[Evaluation of the rate function]\label{Lemma: The rate function}
The rate function defined by \Cref{Eq. Formal definition of the rate function} is equal to $\threshold +1$ in both treaties, i.e.
\begin{equation*}
    \ratefunction[\lcrhittingset] = \ratefunction[\ecomorhittingset] = \threshold +1.
\end{equation*}
\end{lemma}
\begin{proof}
We need to show first that \ratefunction[\lcrhittingset] cannot take any value smaller than or equal to \threshold. Let us assume on the contrary that $\stepfunctionsymbol \in \lcrhittingset \cap \increasingstepset$ such that $\discontinuities =k \leq \threshold$. This means that $\stepfunctionsymbol = \sum_{i \leq k} \maxjump \indicatorfunction[\{\maxjumptime,1\}]$, with $\maxjump[1] \geq \maxjump[2] \geq \dots \maxjump[k] >0$ and $\{ 0,\maxjumptime[1],\maxjumptime[2],\dots,\maxjumptime[k],1 \}$ all distinct. By taking into account \Cref{Assumption: Zero retention level,Assumption: Immediate payment of claims}, we calculate
\begin{align*}
    \lcrmapping
    &= \sup_{t \in [0,1]} \Big\{ \stepfunction  - \levydrift t  - \summaxjump{\threshold}{\stepfunctionsymbol}\Big\}
    = \sup_{t \in [0,1]} \Big\{ \cancel{\sum_{i=1}^k \maxjump \indicatorfunction[\{\maxjumptime,1\}](t)}  - \levydrift t  - \cancel{\sum_{i=1}^k \maxjump \indicatorfunction[\{\maxjumptime,1\}](t)} \Big\}
    = \levydrift_-,
\end{align*}
which states that $\stepfunctionsymbol \not\in \lcrhittingset$ because $\lcrmapping = \levydrift_- \not \geq \crosslevel$. As a result, $\ratefunction[\lcrhittingset] \not = k$, $k \leq r$.

Let us assume now that $\stepfunctionsymbol \in \lcrhittingset \cap \increasingstepset$ such that $\discontinuities = \threshold+1$, i.e.\ $\stepfunctionsymbol = \sum_{i = 1}^{\threshold+1} \maxjump \indicatorfunction[\{\maxjumptime,1\}]$, with $\maxjump[1] \geq \maxjump[2] \geq \dots \maxjump[\threshold+1] >0$ and $\{ 0,\maxjumptime[1],\maxjumptime[2],\dots,\maxjumptime[\threshold+1],1 \}$ all distinct. To calculate \lcrmapping, observe first that
\begin{equation}\label{Eq. Difference between the step function and its r max jumps}
    \stepfunction  - \summaxjump{\threshold}{\stepfunctionsymbol}
    = \sum_{i=1}^{\threshold+1} \maxjump \indicatorfunction[\{\maxjumptime,1\}](t)  -  \summaxjump{\threshold}{\stepfunctionsymbol}
    =
    \begin{cases}
    0, & t< \max \{ \maxjumptime[1],\dots,\maxjumptime[\threshold+1] \}\\
    \maxjump[\threshold+1], & t \geq \max \{ \maxjumptime[1],\dots,\maxjumptime[\threshold+1] \}
    \end{cases},
\end{equation}
because all the claims are ``absorbed'' according to \Cref{Assumption: Immediate payment of claims} before the arrival of the ($\threshold+1$)st claim, which happens at time $t^* =  \max \{ \maxjumptime[1],\dots,\maxjumptime[\threshold+1] \}$. Thus, we can write
\begin{align*}
    \lcrmapping
    &= \sup_{t \in [0,1]} \Big\{ \stepfunction  - \levydrift t  - \summaxjump{\threshold}{\stepfunctionsymbol} \Big\}
    = \sup_{t \in [0,1]} \Big\{ \maxjump[\threshold+1] \prod_{i=1}^{\threshold+1} \indicatorfunction[\{\maxjumptime,1\}](t) - \levydrift t  \Big\}\\
    &= \maxjump[\threshold+1] - \levydrift_+ \max\{ \maxjumptime[1],\dots,\maxjumptime[\threshold+1] \} +\levydrift_- ,
\end{align*}
since $\maxjump[\threshold+1] \prod_{i=1}^{\threshold+1} \indicatorfunction[\{\maxjumptime,1\}](t)$ remains fixed at the value $\maxjump[\threshold+1]$ from $t^*= \max\{ \maxjumptime[1],\dots,\maxjumptime[\threshold+1] \}$ onward, while $-\levydrift t$ decreases or increases depending on the value of \levydrift. Due to the fact that $\stepfunctionsymbol \in \lcrhittingset$, we get
\[
\lcrmapping \geq \crosslevel \Rightarrow \maxjump[\threshold+1] \geq a + \levydrift_+ \max\{ \maxjumptime[1],\dots,\maxjumptime[\threshold+1] \} - \levydrift_- \geq a - \levydrift_- > 0,
\]
i.e.\ $\lcrhittingset \cap \increasingstepset \not = \emptyset$ but contains all step functions with $\threshold+1$ steps such that the ($\threshold+1$)st largest step satisfies: $\maxjump[\threshold+1] \geq a + \levydrift_+ \max\{ \maxjumptime[1],\dots,\maxjumptime[\threshold+1] \} - \levydrift_-$. Thus, $\ratefunction[\lcrhittingset] = \threshold +1$.

The proof for $\ratefunction[\ecomorhittingset] = \threshold +1$ in the ECOMOR treaty is similar. More precisely, it can easily be shown that $\not \exists \ \stepfunctionsymbol \in \ecomorhittingset \cap \increasingstepset$ with $\discontinuities =k \leq \threshold$. Consequently, $\ratefunction[\ecomorhittingset] \not = k$, $k \leq r$. Let us assume next that $\stepfunctionsymbol \in \ecomorhittingset \cap \increasingstepset$ such that $\discontinuities = \threshold+1$, i.e.\ $\stepfunctionsymbol = \sum_{i = 1}^{\threshold+1} \maxjump \indicatorfunction[\{\maxjumptime,1\}]$, with $\maxjump[1] \geq \maxjump[2] \geq \dots \maxjump[\threshold+1] >0$ and $\{ 0,\maxjumptime[1],\maxjumptime[2],\dots,\maxjumptime[\threshold+1],1 \}$ all distinct. It holds that
\begin{equation}\label{Eq. Difference max jumps in ECOMOR}
     \threshold \summaxjump{\threshold+1}{\stepfunctionsymbol} - (\threshold + 1)\summaxjump{\threshold}{\stepfunctionsymbol}
     = -\summaxjump{\threshold}{\stepfunctionsymbol} +
     \begin{cases}
        0, & t< \max \{ \maxjumptime[1],\dots,\maxjumptime[\threshold+1] \}\\
        \threshold \maxjump[\threshold+1], & t \geq \max \{ \maxjumptime[1],\dots,\maxjumptime[\threshold+1] \}
     \end{cases},
\end{equation}
due to \Cref{Assumption: Zero retention level}. By combining now \Cref{Eq. Difference between the step function and its r max jumps,Eq. Difference max jumps in ECOMOR}, we calculate
\begin{align*}
    \ecomormapping
    &= \sup_{t \in [0,1]} \Big\{ \stepfunction  - \levydrift t  -  (\threshold + 1)\summaxjump{\threshold}{\stepfunctionsymbol} + \threshold \summaxjump{\threshold+1}{\stepfunctionsymbol} \Big\}
    = \sup_{t \in [0,1]} \Big\{ (\threshold+1)\maxjump[\threshold+1] \prod_{i=1}^{\threshold+1} \indicatorfunction[\{\maxjumptime,1\}](t) - \levydrift t  \Big\}\\
    &=(\threshold+1)\maxjump[\threshold+1] - \levydrift_+ \max\{ \maxjumptime[1],\dots,\maxjumptime[\threshold+1] \} +\levydrift_- .
\end{align*}
Since $\stepfunctionsymbol \in \ecomorhittingset$, we get $\ecomormapping \geq \crosslevel \Rightarrow (\threshold +1)\maxjump[\threshold+1] \geq \crosslevel + \levydrift_+ \max\{ \maxjumptime[1],\dots,\maxjumptime[\threshold+1] \} - \levydrift_-$, i.e.\ $\ecomorhittingset \cap \increasingstepset \not = \emptyset$ but contains all step functions with $\threshold+1$ steps such that the ($\threshold+1$)st largest step satisfies: $\maxjump[\threshold+1] \geq \big( \crosslevel + \levydrift_+ \max\{ \maxjumptime[1],\dots,\maxjumptime[\threshold+1] \} - \levydrift_- \big)/(\threshold+1) $. Thus, $\ratefunction[\ecomorhittingset] = \threshold +1$, and the proof is complete.
\end{proof}

\begin{remark}\label{Remark: Form of the minimal step function}
The above lemma does not only give the value of the rate function, but it also provides the form of the minimal \stepfunctionsymbol that belongs to the sets \lcrhittingset and \ecomorhittingset, i.e.\ all step functions with $\threshold+1$ steps such that their ($\threshold+1$)st greatest step is greater than or equal to the value $\crosslevel + \levydrift_+ \max\{ \maxjumptime[1],\dots,\maxjumptime[\threshold+1] \} - \levydrift_- $ in the LCR treaty and the value $(\crosslevel + \levydrift_+ \max\{ \maxjumptime[1],\dots,\maxjumptime[\threshold+1] \} - \levydrift_- )/(\threshold+1)$ in the ECOMOR treaty.
\end{remark}

An essential condition of Theorem~3.2 in \cite{rhee2017sample} is that the sets $\closeness{\lcrhittingset} \cap \atmoststeps[{\ratefunction[\lcrhittingset]}]$ and $\closeness{\ecomorhittingset} \cap \atmoststeps[{\ratefunction[\ecomorhittingset]}]$ are bounded away from \atmoststeps[{\ratefunction[\lcrhittingset] - 1}] and \atmoststeps[{\ratefunction[\ecomorhittingset] - 1}], respectively. Verifying this condition allows us then to derive the result \eqref{Eq. Sample path large deviations theorem} for both treaties. We can directly use the value of the rate function in the following result due to \Cref{Lemma: The rate function}.
\begin{lemma}[Bounded away property]
The sets $\closeness{\lcrhittingset} \cap \atmoststeps[\threshold+1]$ and $\closeness{\ecomorhittingset} \cap \atmoststeps[\threshold+1]$ are bounded away from \atmoststeps[\threshold] for some $\smallnumber>0$.
\end{lemma}
\begin{proof}
To simplify the notation in the proof, we write \lcrhittingsetsymbol instead of \lcrhittingset and \ecomorhittingsetsymbol instead of \ecomorhittingset, while the notation \closeness{\lcrhittingsetsymbol}, \closeness{\ecomorhittingsetsymbol} follows naturally.

We start by showing that $\closeness{\lcrhittingsetsymbol} \cap \atmoststeps[\threshold+1]$ is bounded away from \atmoststeps[\threshold] for some $\smallnumber>0$. Thanks to \Cref{Lemma. Useful inequality about the sum of the max jump}, we have that $\closeness{\lcrhittingsetsymbol} \subset \lcrhittingsetsymbol(\smallnumber)$, where $\lcrhittingsetsymbol(\smallnumber) = \lcrmappingsymbol^{-1}\big( [\crosslevel - (\abs[\levydrift] + 2 \threshold + 1)\smallnumber ,\infty) \big)$. Hence, it suffices to show that $\lcrhittingsetsymbol(\smallnumber) \cap \atmoststeps[\threshold+1]$ is bounded away from \atmoststeps[\threshold]. Let $\stepfunctionsymbol \in \lcrhittingsetsymbol(\smallnumber) \cap \atmoststeps[\threshold+1]$. Since $\stepfunctionsymbol \in \atmoststeps[\threshold+1]$, we can write $\stepfunctionsymbol = \sum_{i = 1}^{\threshold+1} \maxjump \indicatorfunction[\{\maxjumptime,1\}]$ with $\maxjump[1] \geq \maxjump[2] \geq \dots \geq \maxjump[\threshold+1] \geq 0$, for which it holds that $\lcrmapping \leq \maxjump[\threshold+1] - \levydrift_+ \max\{ \maxjumptime[1],\dots,\maxjumptime[\threshold+1] \}+\levydrift_- \leq \maxjump[\threshold+1] + \levydrift_-$ according to the proof of \Cref{Lemma: The rate function}. Furthermore, $\stepfunctionsymbol \in \lcrhittingsetsymbol(\smallnumber) \Leftrightarrow \lcrmapping \geq \crosslevel - (\abs[\levydrift] + 2 \threshold + 1)\smallnumber$. Combining the two inequalities, we obtain that $\maxjump[\threshold+1] \geq (\crosslevel - \levydrift_-) - (\abs[\levydrift] + 2 \threshold + 1)\smallnumber \geq (\crosslevel-\levydrift_-)/2$, for $\smallnumber \leq (\crosslevel-\levydrift_-)/2 (\abs[\levydrift] + 2 \threshold + 1)$. In other words, for $\smallnumber \leq (\crosslevel-\levydrift_-)/2 (\abs[\levydrift] + 2 \threshold + 1)$, $\stepfunctionsymbol \in \lcrhittingsetsymbol(\smallnumber) \cap \exactsteps[\threshold+1] \subset \lcrhittingsetsymbol(\smallnumber) \cap \atmoststeps[\threshold+1]$ with jump sizes bounded from below by $(\crosslevel-\levydrift_-)/2$, which implies that $\lcrhittingsetsymbol(\smallnumber) \cap \atmoststeps[\threshold+1]$ is bounded away from \atmoststeps[\threshold] .

In a similar manner, it suffices to show that $\ecomorhittingsetsymbol(\smallnumber) \cap \atmoststeps[\threshold+1]$ is bounded away from \atmoststeps[\threshold], where $\ecomorhittingsetsymbol(\smallnumber) = \ecomormappingsymbol^{-1}\big( [\crosslevel - (\abs[\levydrift] +  4 \threshold^2 + 4 \threshold + 1)\smallnumber ,\infty) \big)$. Let $\stepfunctionsymbol \in \ecomorhittingsetsymbol(\smallnumber) \cap \atmoststeps[\threshold+1]$. Since $\stepfunctionsymbol \in \atmoststeps[\threshold+1]$, we can write $\stepfunctionsymbol = \sum_{i = 1}^{\threshold+1} \maxjump \indicatorfunction[\{\maxjumptime,1\}]$ with $\maxjump[1] \geq \maxjump[2] \geq \dots \geq \maxjump[\threshold+1] \geq 0$, for which it holds that $\ecomormapping \leq (\threshold+1)\maxjump[\threshold+1] - \levydrift_+ \max\{ \maxjumptime[1],\dots,\maxjumptime[\threshold+1] \} + \levydrift_- \leq (\threshold+1)\maxjump[\threshold+1] + \levydrift_-$. Furthermore, $\stepfunctionsymbol \in \ecomorhittingsetsymbol(\smallnumber) \Leftrightarrow \ecomormapping \geq \crosslevel - (\abs[\levydrift] +  4 \threshold^2 + 4 \threshold + 1)\smallnumber$. Combining the two inequalities, we obtain that $(\threshold+1)\maxjump[\threshold+1] \geq (\crosslevel-\levydrift_-) - (\abs[\levydrift] +  4 \threshold^2 + 4 \threshold + 1)\smallnumber \geq (\crosslevel-\levydrift_-)/2$, for $\smallnumber \leq (\crosslevel-\levydrift_-)/2 (\abs[\levydrift] +  4 \threshold^2 + 4 \threshold + 1)$. In other words, the jump sizes of \stepfunctionsymbol are bounded from below by $(\crosslevel-\levydrift_-)/2(\threshold+1)$, which implies that $\ecomorhittingsetsymbol(\smallnumber) \cap \atmoststeps[\threshold+1]$ is bounded away from \atmoststeps[\threshold] for $\smallnumber \leq (\crosslevel-\levydrift_-)/2 (\abs[\levydrift] + 4 \threshold^2 + 4 \threshold + 1)$, and the proof is complete.
\end{proof}

Let \LCRpreconstant:= \preconstant[\lcrhittingset] and \ECOMORpreconstant:= \preconstant[\ecomorhittingset]. According to Section~3.1 in \cite{rhee2017sample}, the $\liminf$ and $\limsup$ in \Cref{Eq. Sample path large deviations theorem} yield the same result when
\begin{equation*}
    \lowerconstant = \limitconstant= \upperconstant.
\end{equation*}
However, the above equality holds when the set \randomset is \measureconstant[\ratefunction]-continuous, i.e.\ $\measureconstant[\ratefunction](\boundaryset)=0$, where the boundary $\boundaryset = \closedset \setminus \openset$ of a set \randomset is the closure of \randomset without its interior. We prove in the next lemma that the sets \lcrhittingset and \ecomorhittingset are both \measureconstant[\threshold+1]-continuous.

\begin{lemma}[Equality of the limits]
The sets \lcrhittingset and \ecomorhittingset are \measureconstant[\threshold+1]-continuous, i.e.\ $\preconstant[{\boundaryset[\lcrhittingset]}] = \preconstant[{\boundaryset[\ecomorhittingset]}] = 0$.
\end{lemma}
\begin{proof}
To simplify the notation in the proof, we write again \lcrhittingsetsymbol instead of \lcrhittingset and \ecomorhittingsetsymbol instead of \ecomorhittingset, while the notation \openset[\lcrhittingsetsymbol], \openset[\ecomorhittingsetsymbol], \closedset[\lcrhittingsetsymbol], \closedset[\ecomorhittingsetsymbol]  follows naturally.

We start by showing the \measureconstant[\threshold+1]-continuity of \lcrhittingsetsymbol. In compliance with the notation introduced in \Cref{Section: Preliminaries on the Skorokhod topology and notation}, we consider the function $\inversestepmappingsymbol[\threshold+1]: \exactsteps \to \coupledset$ such that
\begin{align*}
    \inversestepmapping[\threshold+1]{\openset[\lcrhittingsetsymbol]}
    &= \inversestepmapping[\threshold+1]{\lcrmappingsymbol^{-1}\big( (\crosslevel,\infty) \big)}
    = \Big\{ (\vect{\maxjumpsymbol}, \vect{\maxjumptimesymbol}) \in \coupledset : \maxjump[\threshold+1] > \crosslevel + \levydrift_+ \max \{ \maxjumptime[1],\dots,\maxjumptime[\threshold+1] \} - \levydrift_- \Big\}, \\
    \inversestepmapping[\threshold+1]{\closedset[\lcrhittingsetsymbol]}
    &= \inversestepmapping[\threshold+1]{\lcrmappingsymbol^{-1}\big( [\crosslevel,\infty) \big)}
    = \Big\{ (\vect{\maxjumpsymbol}, \vect{\maxjumptimesymbol}) \in \coupledset : \maxjump[\threshold+1] \geq \crosslevel + \levydrift_+ \max \{ \maxjumptime[1],\dots,\maxjumptime[\threshold+1] \} - \levydrift_- \Big\}.
\end{align*}
Obviously, the set $\inversestepmapping[\threshold+1]{\closedset[\lcrhittingsetsymbol]} \setminus \inversestepmapping[\threshold+1]{\openset[\lcrhittingsetsymbol]}$ has zero Lebesgue measure. Combining this with $\openset[\lcrhittingsetsymbol] \subseteq \lcrhittingsetsymbol \subseteq \closedset[\lcrhittingsetsymbol]$ and \lcrmappingsymbol being a continuous function, we conclude that $\preconstant[{\boundaryset[\lcrhittingsetsymbol]}] = 0$, i.e.\ \lcrhittingsetsymbol is \measureconstant[\threshold+1]-continuous.
To prove the \measureconstant[\threshold+1]-continuity of \ecomorhittingsetsymbol, it suffices to observe that the set $\inversestepmapping[\threshold+1]{\closedset[\ecomorhittingsetsymbol]} \setminus \inversestepmapping[\threshold+1]{\openset[\ecomorhittingsetsymbol]}$ has zero Lebesgue measure, where
\begin{align*}
    \inversestepmapping[\threshold+1]{\openset[\ecomorhittingsetsymbol]}
    &= \inversestepmapping[\threshold+1]{\ecomormappingsymbol^{-1}\big( (\crosslevel,\infty) \big)}\\
    &= \Big\{ (\vect{\maxjumpsymbol}, \vect{\maxjumptimesymbol}) \in \coupledset : \maxjump[\threshold+1] > \big(\crosslevel + \levydrift_+ \max\{ \maxjumptime[1],\dots,\maxjumptime[\threshold+1] \} - \levydrift_- \big)/(\threshold+1) \Big\}, \\
    \inversestepmapping[\threshold+1]{\closedset[\ecomorhittingsetsymbol]}
    &= \inversestepmapping[\threshold+1]{\ecomormappingsymbol^{-1}\big( [\crosslevel,\infty) \big)}\\
    &= \Big\{ (\vect{\maxjumpsymbol}, \vect{\maxjumptimesymbol}) \in \coupledset : \maxjump[\threshold+1] \geq \big(\crosslevel + \levydrift_+ \max\{ \maxjumptime[1],\dots,\maxjumptime[\threshold+1] \} - \levydrift_- \big)/(\threshold+1) \Big\},
\end{align*}
which follows by the same reasoning.
\end{proof}

We calculate now the pre\-/constants \limitconstant[\lcrhittingset] and \limitconstant[\ecomorhittingset].

\begin{lemma}[Calculation of the pre-constant]\label{Lemma: The pre-constant}
The constants \LCRpreconstant and \ECOMORpreconstant are given by
\begin{alignat*}{2}
    \LCRpreconstant
    =&  \quad \frac{1}{(\threshold+1)!}
    &&\times
    \begin{cases}
    \crosslevel^{-(\threshold + 1) \powerindex} \cdot \prescript{}{2}{F}_1[\threshold+1, (\threshold + 1) \powerindex; \threshold + 2; -\levydrift/\crosslevel], & \levydrift>0, \\
    (\crosslevel+\levydrift)^{-(\threshold+1)\powerindex}, & \levydrift<0.
    \end{cases}
    , \\
    \ECOMORpreconstant
    =& \frac{(\threshold+1)^{(\threshold+1)\powerindex}}{(\threshold+1)!}
    &&\times
    \begin{cases}
    \crosslevel^{-(\threshold + 1) \powerindex} \cdot \prescript{}{2}{F}_1[\threshold+1, (\threshold + 1) \powerindex; \threshold + 2; -\levydrift/\crosslevel], & \levydrift>0, \\
    (\crosslevel+\levydrift)^{-(\threshold+1)\powerindex}, & \levydrift<0.
    \end{cases}
\end{alignat*}
\end{lemma}
\begin{proof}
Recall that \LCRpreconstant:= \preconstant[\lcrhittingset] and \ECOMORpreconstant:= \preconstant[\ecomorhittingset]. To calculate these constants, we use the definition of the measure \preconstant[\bullet] in \Cref{Eq. The measure of the constants}. We start with \LCRpreconstant. It is known that for $\uniform[1],\dots,\uniform[\threshold+1] \sim \uniformdistribution(0,1)$, the distribution of the r.v.\ $\max \{ \uniform[1],\dots,\uniform[\threshold+1] \}$ is given by the formula $\pr(\max \{ \uniform[1],\dots,\uniform[\threshold+1] \leq t) = t^{\threshold+1}$. Furthermore, by using that $\int_b^{+\infty} \powerindex y^{-n \powerindex - 1} dy = b^{-n \powerindex}/n$ with $b>0$, we calculate recursively the following multiple integrals for $n \in \naturals$ and positive $y_i$'s
\begin{align*}
    \recursive
    &= \int\limits_{y_1 \geq \dots \geq y_{n+1}} \prod_{i=1}^n \powerindex y_i^{-\powerindex-1}
     dy_1 \dots dy_n \\
    &= \int\limits_{y_n = y_{n+1}}^{+\infty} \int\limits_{y_{n-1} = y_n}^{+\infty} \dots \int\limits_{y_2 = y_3}^{+\infty} \prod_{i=2}^{n} \powerindex y_i^{-\powerindex-1}
    \underbrace{\Bigg(
    \int\limits_{y_1 = y_2}^{+\infty} \powerindex y_1^{-\powerindex-1} dy_1 \Bigg)}_{= y_2^{-\powerindex}} dy_2 \dots dy_n \\
    &= \int\limits_{y_n = y_{n+1}}^{+\infty} \dots \int\limits_{y_3 = y_4}^{+\infty} \prod_{i=3}^{n} \powerindex y_i^{-\powerindex-1}
    \underbrace{\Bigg( \int\limits_{y_2 = y_3}^{+\infty}  \powerindex y_2^{-2\powerindex-1} dy_2 \Bigg)}_{= y_3^{-2 \powerindex}/2} dy_3 \dots dy_n
    = \dots = \frac{1}{n!}(y_{n+1})^{- n \powerindex}.
\end{align*}
Consequently, in case $\levydrift>0$, we obtain by virtue of \Cref{Remark: Form of the minimal step function}
\begin{align*}
    \LCRpreconstant
    &= \e \big[ \measurepowerexponent{\threshold+1} \{ \vect{y} \in \reals[\threshold+1]_+ : \sum_{i=1}^{\threshold+1} y_i \indicatorfunction[\{\uniform,1\}] \in \lcrhittingset \} \big] \\
    &= \e \left[ \int\limits_{y_1 \geq \dots \geq y_{\threshold+1}>0} \prod_{i=1}^{\threshold+1} \powerindex y_i^{-\powerindex-1}
    \indicatorfunction[{ \{ y_{\threshold+1} \geq \crosslevel + \levydrift \max \{ \uniform[1],\dots,\uniform[\threshold+1]\} \} }] dy_1 \dots dy_{\threshold+1} \right] \\
    &= \int\limits_{t \in [0,1]} \int\limits_{y_1 \geq \dots \geq y_{\threshold+1}>0} \prod_{i=1}^{\threshold+1} \powerindex y_i^{-\powerindex-1}
    \indicatorfunction[\{y_{\threshold+1} \geq \crosslevel + \levydrift t \}\}] (\threshold+1) t^\threshold dy_1 \dots dy_{\threshold+1} dt\\
    &= \int\limits_{t \in [0,1]} \int\limits_{ y_{\threshold+1}>0} \recursive[\threshold] \powerindex y_{\threshold+1}^{-\powerindex-1}
    \indicatorfunction[\{y_{\threshold+1} \geq \crosslevel + \levydrift t \} \}] (\threshold+1) t^\threshold dy_{\threshold+1} dt\\
    &= \int\limits_{t \in [0,1]} \int\limits_{ y_{\threshold+1}=\crosslevel + \levydrift t }^{+\infty} \frac{1}{\threshold!}(y_{\threshold+1})^{- \threshold \powerindex} \powerindex y_{\threshold+1}^{-\powerindex-1}
    (\threshold+1) t^\threshold dy_{\threshold+1} dt\\
    &= \frac{\threshold+1}{\threshold!}
    \int\limits_0^1 t^\threshold \Bigg( \int\limits_{\crosslevel + \levydrift t }^{+\infty} \powerindex (y_{\threshold+1})^{-(\threshold+1)\powerindex-1} dy_{\threshold+1} \Bigg) dt
    = \frac{1}{\threshold!} \int\limits_0^1 t^\threshold (\crosslevel + \levydrift t)^{-(\threshold+1)\powerindex} dt\\
    &=\frac{\crosslevel^{-(\threshold + 1) \powerindex}}{(\threshold+1)!} \cdot \prescript{}{2}{F}_1[\threshold+1, (\threshold + 1) \powerindex; \threshold + 2; -\levydrift/\crosslevel].
\end{align*}
Analogously, we find
\begin{align*}
    \ECOMORpreconstant
    &= \e \big[ \measurepowerexponent{\threshold+1} \{ \vect{y} \in \reals[\threshold+1]_+ : \sum_{i=1}^{\threshold+1} y_i \indicatorfunction[\{\uniform,1\}] \in \ecomorhittingset \} \big] \\
    &=\int\limits_{y_1 \geq \dots \geq y_{\threshold+1}>0} \prod_{i=1}^{\threshold+1} \powerindex y_i^{-\powerindex-1}
    \indicatorfunction[{ \{y_{\threshold+1} \geq (\crosslevel + \levydrift \max \{ \uniform[1],\dots,\uniform[\threshold+1]\} )/(\threshold+1) \} }] dy_1 \dots dy_{\threshold+1} \\
    &= \int\limits_{t \in [0,1]} \int\limits_{ y_{\threshold+1}>0} \recursive[\threshold] \powerindex y_{\threshold+1}^{-\powerindex-1}
    \indicatorfunction[\{y_{\threshold+1} \geq (\crosslevel + \levydrift t)/(\threshold+1) \} \}] (\threshold+1) t^\threshold dy_{\threshold+1} dt\\
    &= \frac{1}{\threshold!} \int\limits_0^1 t^\threshold \left( \frac{\crosslevel + \levydrift t}{\threshold+1} \right)^{-(\threshold+1)\powerindex} dt
    = (\threshold+1)^{(\threshold+1)\powerindex} \frac{1}{\threshold!} \int\limits_0^1 t^\threshold ( \crosslevel + \levydrift t )^{-(\threshold+1)\powerindex} dt\\
    &= (\threshold+1)^{(\threshold+1)\powerindex}  \frac{\crosslevel^{-(\threshold + 1) \powerindex}}{(\threshold+1)!} \cdot \prescript{}{2}{F}_1[\threshold+1, (\threshold + 1) \powerindex; \threshold + 2; -\levydrift/\crosslevel].
\end{align*}
In case $\levydrift<0$, the coefficients simplify to
\begin{align*}
    \LCRpreconstant
    &= \e \left[ \int\limits_{y_1 \geq \dots \geq y_{\threshold+1}>0} \prod_{i=1}^{\threshold+1} \powerindex y_i^{-\powerindex-1}
    \indicatorfunction[\{y_{\threshold+1} \geq \crosslevel + \levydrift \}] dy_1 \dots dy_{\threshold+1} \right]
    = \int\limits_{ \crosslevel + \levydrift}^{+\infty} \recursive[\threshold] \powerindex y_{\threshold+1}^{-\powerindex-1} dy_{\threshold+1}\\
    &= \ldots = \frac{1}{(\threshold+1)!} (\crosslevel+\levydrift)^{-(\threshold+1)\powerindex}, \quad \text{and,}\\
    \ECOMORpreconstant
    &= (\threshold+1)^{(\threshold+1)\powerindex} \frac{1}{(\threshold+1)!} (\crosslevel+\levydrift)^{-(\threshold+1)\powerindex}.
\end{align*}
\end{proof}

\begin{remark}\label{Remark: Alternative expression for the preconstants}
When $\levydrift>0$, the coefficients \LCRpreconstant and \ECOMORpreconstant can be equivalently expressed in terms of finite sums involving the Gamma function. More precisely, by applying \threshold times integration by parts, we calculate for $k>\threshold+1$ that
\begin{align*}
     \frac{1}{\threshold!} \int t^\threshold (\crosslevel + \levydrift t)^{-k} dt
     =& \sum_{m=1}^{\threshold+1} \frac{(-1)^{m+1} t^{\threshold+1-m}}{(\threshold+1-m)!} \frac{(\crosslevel + \levydrift t)^{m-k}}{\levydrift^m \prod_{j=1}^m(j-k)}
     = \sum_{m=1}^{\threshold+1} \frac{(-1)^{m+1} t^{\threshold+1-m}}{(\threshold+1-m)!} \frac{(\crosslevel + \levydrift t)^{m-k}}{\levydrift^m (1-k)_m} \\
     =& -\sum_{m=1}^{\threshold+1} \frac{ t^{\threshold+1-m}}{(\threshold+1-m)!} \frac{(\crosslevel + \levydrift t)^{m-k}}{\levydrift^m (k-m)_m} \Rightarrow \\
     \frac{1}{\threshold!} \int\limits_0^1 t^\threshold (\crosslevel + \levydrift t)^{-k} dt
     =& \frac{  \crosslevel^{\threshold+1-k} }{\levydrift^{\threshold+1} (k-\threshold-1)_{\threshold+1}} -\sum_{m=1}^{\threshold+1}
     \frac{(\crosslevel + \levydrift)^{m-k}}{(\threshold+1-m)! \levydrift^m (k-m)_m},
\end{align*}
where $(b)_k = \Gamma(b+k)/\Gamma(b)$ is again the Pochhammer symbol. Thus,

\begin{align*}
    \LCRpreconstant
    &=  \frac{ \crosslevel^{-(\threshold+1)(\powerindex-1)} \Gamma \big( (\threshold+1) \powerindex \big) }{\levydrift^{\threshold+1} \Gamma \big( (\threshold+1)(\powerindex-1) \big)} -\sum_{m=1}^{\threshold+1} \frac{(\crosslevel + \levydrift)^{m-(\threshold+1)\powerindex} \Gamma \big( (\threshold+1) \powerindex \big) }{(\threshold+1-m)! \levydrift^m \Gamma \big( ( \threshold+1)\powerindex - m \big)}, \\
    \ECOMORpreconstant
    &= (\threshold+1)^{(\threshold+1)\powerindex} \left( \frac{ \crosslevel^{-(\threshold+1)(\powerindex-1)} \Gamma \big( (\threshold+1) \powerindex \big) }{\levydrift^{\threshold+1} \Gamma \big( (\threshold+1)(\powerindex-1) \big)} -\sum_{m=1}^{\threshold+1} \frac{(\crosslevel + \levydrift)^{m-(\threshold+1)\powerindex} \Gamma \big( (\threshold+1) \powerindex \big) }{(\threshold+1-m)! \levydrift^m \Gamma \big( ( \threshold+1)\powerindex - m \big)} \right).
\end{align*}
\end{remark}

\begin{remark}\label{Remark: Asymptotic approximation}
In the absence of reinsurance ($\threshold = 0$), the pre\-/constant simplifies to
\begin{equation*}
    \frac{\crosslevel^{-\powerindex +1} - (\crosslevel + \levydrift)^{-\powerindex +1}}{\levydrift (\powerindex - 1)},
\end{equation*}
which can also be derived from existing results, see e.g.\ \cite{asmussen1996large,embrechts1979subexponentiality}.
\end{remark}

Finally, we know from \cite{kyprianou-ILFLPA} that the compound Poisson aggregate claim process $\aggregateclaims = \sum_{i=1}^{\poissonprocess} \claim$ is a special L\'{e}vy process with L\'{e}vy measure $\levymeasure(dx) = \poissonrate \claimsizedis[dx]$, which means that $n \cdot \levymeasure[n,\infty) = \poissonrate n \comclaimsizedis[n] = \poissonrate \slowlyvarying[n] n^{-\powerindex+1}$, $n \in \naturals$. We conclude the proof by combining this result with \Cref{Lemma: The pre-constant} to obtain the expression \eqref{Eq. Expression of main result}.

\section{Numerical implementations}\label{Section: Numerical implementations}
Our primary goal in this section is to verify our asymptotic approximations in \Cref{Theorem: Asymptotic ruin probability} via numerical illustration.
For this purpose, we employ an importance sampling scheme that was developed in \cite{bohan2017efficient} and it is proved to be strongly efficient in the current setting.
We provide a short description of this scheme in \Cref{Appendix}.

We use a shifted Pareto distribution for the claim sizes, i.e.\ $\comclaimsizedis = (x+1)^{-\powerindex}$, $x\geq 0$, and $\e \claim[] = 1/(\powerindex-1)$.
In addition, we calculate the net premiums $\cedentpremium^{L} = \premium - \reinsurerpremium^{L}$ and $\cedentpremium^{E} = \premium - \reinsurerpremium^{E}$ of the insurer after purchasing an LCR or ECOMOR reinsurance for a premium $\reinsurerpremium^{L}$ and $\reinsurerpremium^{E}$, respectively.

We assume here that the reinsurance premiums are determined according to an {\it expected value principle}, see e.g.\ \cite{albrecher-RASA}. Hence, we need to determine $\e \reinsurance$. As the Pareto claims arrive according to a Poisson process with rate \poissonrate, we follow \cite{berliner1972correlations} to obtain 
\begin{align*}
    \e \LCR{\threshold} &= (\poissonrate t)^{1/\powerindex} \sum_{i=1}^\threshold \frac{\gamma(i-1/\powerindex,\poissonrate t)}{\Gamma(i)} - \sum_{i=1}^\threshold \frac{\gamma(i,\poissonrate t)}{\Gamma(i)},\\
    \e \ECOMOR{\threshold} &= (\poissonrate t)^{1/\powerindex} \Bigg( \sum_{i=1}^\threshold \frac{\gamma(i-1/\powerindex,\poissonrate t)}{\Gamma(i)} - \threshold \frac{\gamma(\threshold+1-1/\powerindex,\poissonrate t)}{\Gamma(\threshold+1)} \Bigg) 
    -\Bigg( \sum_{i=1}^\threshold \frac{\gamma(i,\poissonrate t)}{\Gamma(i)} - \threshold \frac{\gamma(\threshold+1,\poissonrate t)}{\Gamma(\threshold+1)} \Bigg),
\end{align*}
where $\displaystyle \gamma(k,s) = \int_0^s e^{-u} u^{k-1}du$ is the lower incomplete gamma function. Thus, if $\insurersafetyloading, \reinsurersafetyloading>0$ are the relative safety loadings imposed by the insurer and reinsurer, respectively, we calculate the annual retained premium \cedentpremium over a period of $n$ years via the formula $\cedentpremium = (1+\insurersafetyloading) \e \aggregateclaims[1] - (1 + \reinsurersafetyloading) \e \reinsurance[n]/n$. Correspondingly,
\begin{align*}
    \cedentpremium^{L} =& (1+\insurersafetyloading)  \frac{\poissonrate}{\powerindex-1} - (1+\reinsurersafetyloading) \Bigg( (\poissonrate n)^{1/\powerindex} \sum_{i=1}^\threshold \frac{\gamma(i-1/\powerindex,\poissonrate n)}{\Gamma(i)} - \sum_{i=1}^\threshold \frac{\gamma(i,\poissonrate n)}{\Gamma(i)} \Bigg) \Biggm/ n,\\
    \cedentpremium^{E} =&  (1+\insurersafetyloading)  \frac{\poissonrate}{\powerindex-1} - (1+\reinsurersafetyloading) (\poissonrate n)^{1/\powerindex} \Bigg( \sum_{i=1}^\threshold \frac{\gamma(i-1/\powerindex,\poissonrate n)}{\Gamma(i)} - \threshold \frac{\gamma(\threshold+1-1/\powerindex,\poissonrate n)}{\Gamma(\threshold+1)} \Bigg) \Bigg/ n \\
    &+(1+\reinsurersafetyloading)\Bigg( \sum_{i=1}^\threshold \frac{\gamma(i,\poissonrate n)}{\Gamma(i)} - \threshold \frac{\gamma(\threshold+1,\poissonrate n)}{\Gamma(\threshold+1)} \Bigg) \Bigg/ n.
\end{align*}

We fix now $n=20$, $\powerindex = 1.5$, $\poissonrate = 10$, $\insurersafetyloading=0.2$, $\reinsurersafetyloading=0.3$ (safety loadings for reinsurance are typically larger than for primary insurance, see \cite{albrecher-RASA}) to obtain the following figures:

\begin{table}[h]
\centering
\begin{adjustbox}{max width=\textwidth}
  \begin{tabular}{|c|cccccc|}\hline
 \threshold &  $\reinsurerpremium^{L}$ & $\reinsurerpremium^{E}$ & $\cedentpremium^{L}$ & $\cedentpremium^{E}$ & $\levydrift_L$  & $\levydrift_E$  \\ \hline
 0 & 0 & 0  & 24 & 24  & 4 & 4 \\
 1 & 4.5309   & 3.0539  & 18.1098 & 20.0299  & -1.8902  & \ 0.0298  \\
 2 & 6.0078   & 4.0719  & 16.1897 & 18.7065  & -3.8102  & -1.2935 \\
 3 & 6.9758   & 4.7505  & 14.9314 & 17.8242  & -5.0686  & -2.1757  \\ \hline
\end{tabular}
\end{adjustbox}
  \caption{Premiums for LCR and ECOMOR treaties for varying \threshold for $n = 20$, $\poissonrate=10$, $\powerindex = 1.5$, $\insurersafetyloading=0.2$, and $\reinsurersafetyloading=0.3$.}\label{Table: The different thresholds}
\end{table}

Finally, we choose the values of \crosslevel such that the asymptotic approximations for LCR and ECOMOR are simultaneously defined. In other words, it should hold that $\crosslevel > \max \{ -\levydrift_L, -\levydrift_E, 0 \}$, where $\levydrift_i = \cedentpremium^i -  \poissonrate/(\powerindex-1)$, $i \in \{ L,E \}$. It is clear from \Cref{Table: The different thresholds} that $\levydrift_L < \levydrift_E$. Therefore, both approximations are simultaneously valid for $\crosslevel > \max \{ -\levydrift_L, 0 \}$.

The results under both LCR and ECOMOR treaties for different combinations of \threshold and \crosslevel are presented in \Cref{Fig. Numerics1,Fig. Numerics2,Fig. Numerics3}. We plot the simulation estimates (circles) together with the large deviation approximation (line) of the rare event probabilities as a function of $n$. Note that the results for $r=0$ can be considered as a sanity check for our simulation study.

We observe that the large deviation results become accurate as $n$ grows, in line with \Cref{Theorem: Asymptotic ruin probability}. It is quite remarkable that in most cases the resulting approximation is already excellent for $n=20$. This corresponds to a time horizon of 20 years for the present insurance application. For fixed $n$, the quality of the asymptotic approximation improves as \crosslevel increases. Finally, we recognize that LCR always leads to lower ruin probabilities than ECOMOR, which is intuitively expected. However, the explicit expression given in \Cref{Theorem: Asymptotic ruin probability}, allows for the first time to quantitatively assess the effects of the model parameters on the resulting ruin probabilities.

\begin{figure}[ht]
	\centering
  \includegraphics[width=0.8\textwidth]{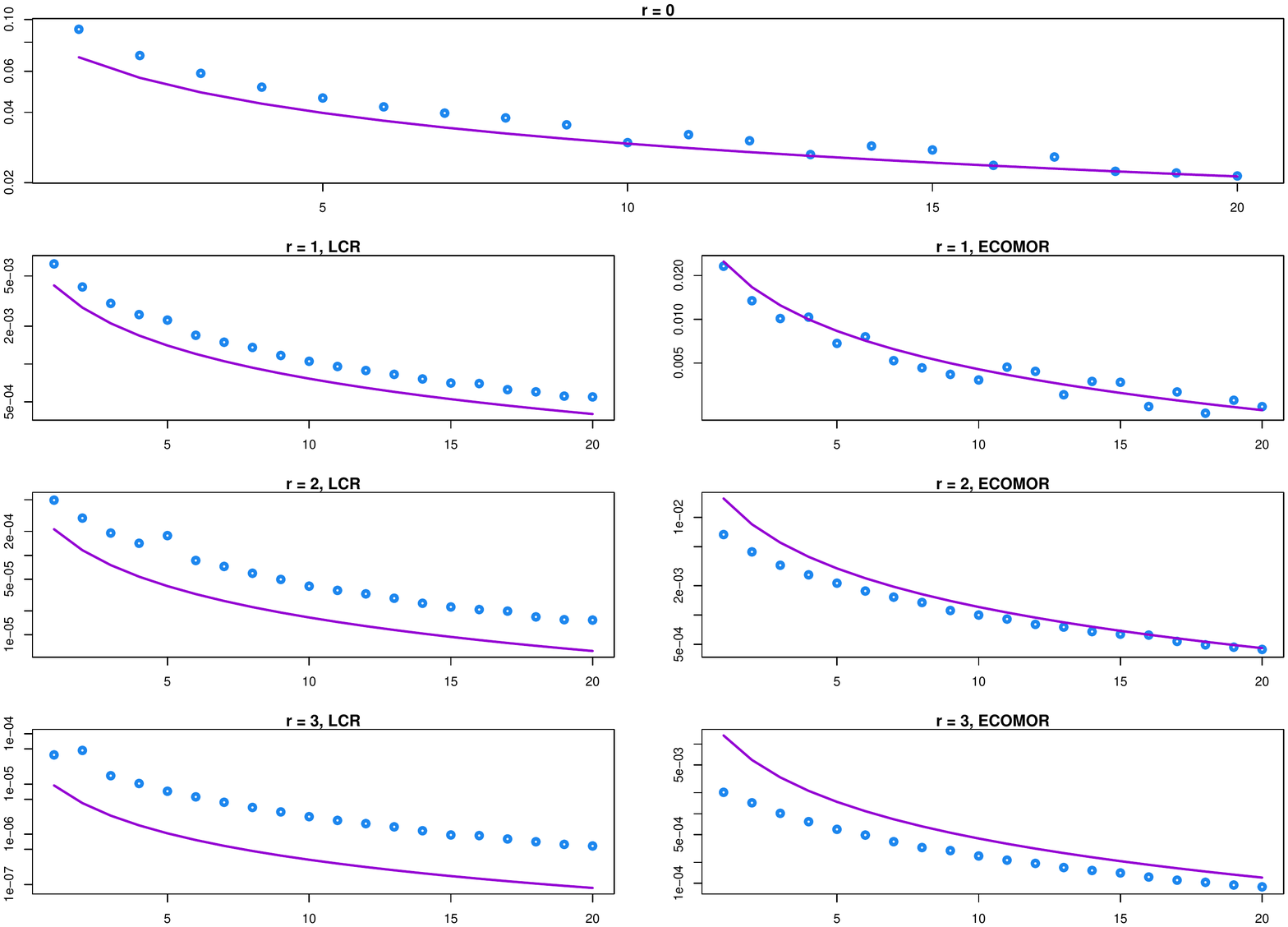}
	\caption{Numerical results for both LCR and ECOMOR treaties, for $\crosslevel=20$.}
	\label{Fig. Numerics1}
\end{figure}

\begin{figure}[h]
	\centering
  \includegraphics[width=0.8\textwidth]{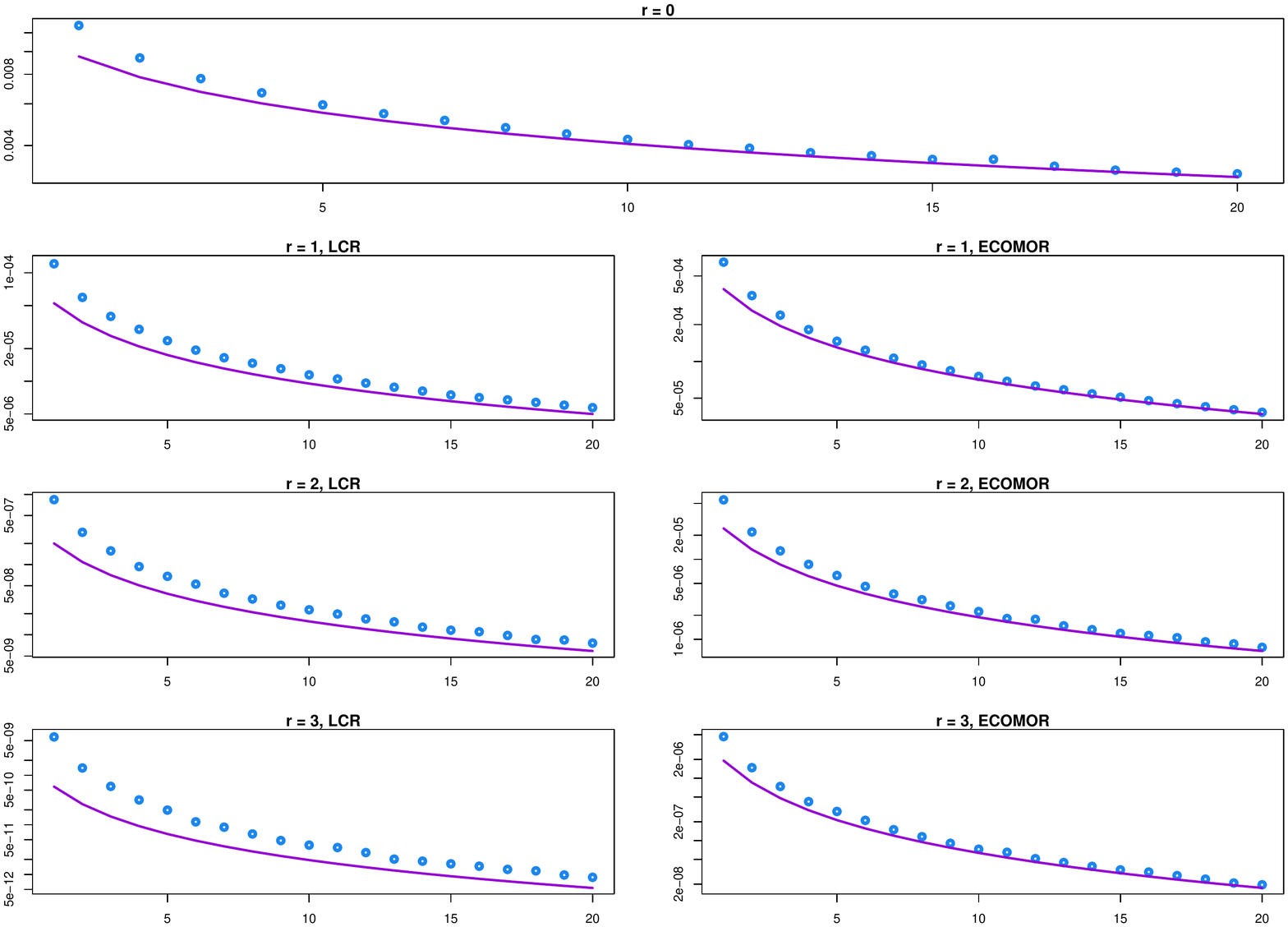}
	\caption{Numerical results for both LCR and ECOMOR treaties, for $\crosslevel=80$.}
	\label{Fig. Numerics2}
\end{figure}

\begin{figure}[ht]
	\centering
  \includegraphics[width=0.8\textwidth]{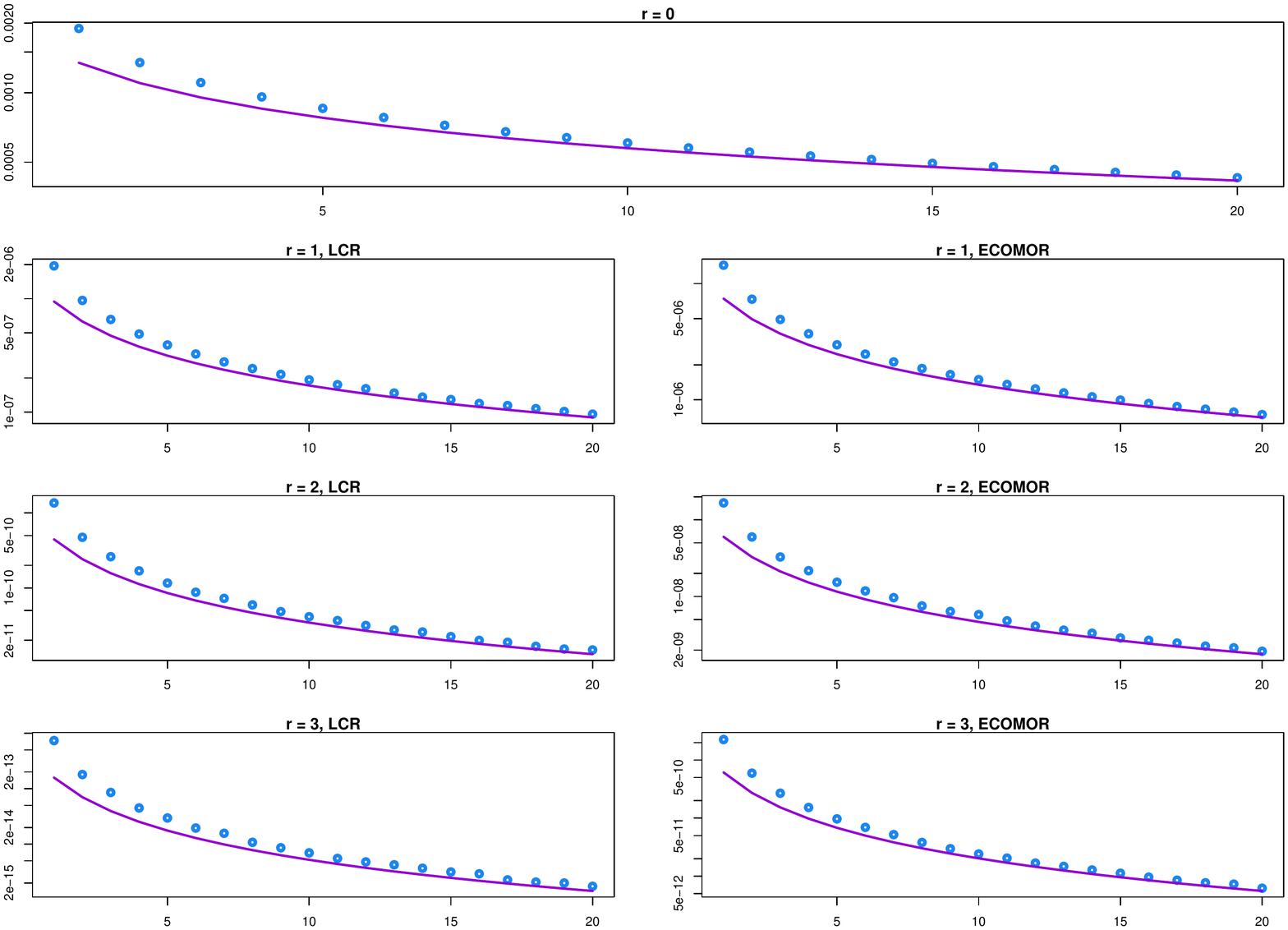}
	\caption{Numerical results for both LCR and ECOMOR treaties, for $\crosslevel=300$.}
	\label{Fig. Numerics3}
\end{figure}

\appendix
\section{Appendix: Short description of the simulation technique}\label{Appendix}

Our simulation estimator is based on an importance sampling strategy; see e.g.\ Chapter~V of \cite{asmussen-SS}.
To be precise, for $\smallnumber>0$, we define the auxiliary set
\begin{equation*}
    B_\smallnumber
    =
    \{ \stepfunctionsymbol\in\skorokhodspace \colon \discontinuitiesset[\stepfunctionsymbol,\smallnumber] \geq \threshold+1 \},
\end{equation*}
where \discontinuitiesset[\stepfunctionsymbol,\smallnumber] is given in \eqref{Eq. definition of set of disc. of magnitude at least epsilon}.
We propose an importance distribution $\mathds Q_{\smallnumber,w}$ that is determined by
\begin{equation*}
    \mathds Q_{\smallnumber,w} (\,\bullet\,)
    =
    w\pr (\,\bullet\,) + (1-w)\mathds Q_{\smallnumber} (\,\bullet\,),
\end{equation*}
where $w\in(0,1)$ and $\mathds Q_{\smallnumber} (\,\bullet\,) = \pr (\,\bullet\, |\, \scaledaggregateclaimssequence \in B_\smallnumber )$.
Note that $\mathds Q_{\smallnumber} (\,\bullet\,)$ is the conditional distribution given the event \scaledaggregateclaimssequence has at least $\threshold+1$ discontinuities of magnitude \smallnumber.
The proposed importance distribution has the following interpretation.
We flip a coin at the beginning of each simulation.
We generate with probability $w$ the sample path of \scaledaggregateclaimssequence under the original measure and with probability $1-w$, we sample \scaledaggregateclaimssequence under the measure $\mathds Q_{\smallnumber} (\,\bullet\,)$.
To compensate for the bias introduced by the importance distribution, a likelihood ratio -- that is the Radon\-/Nikodym derivative between \pr and $\mathds Q_{\smallnumber,w}$ -- must be included in the estimator.
In our case, the estimator $Z_n$ for $\pr(\scaledaggregateclaimssequence \in \randomset )$ is then given by
\begin{equation*}
Z_n
=
\indicatorfunction[\{\scaledaggregateclaimssequence \in \randomset \}] \frac{d\pr}{d\mathds Q_{\smallnumber,w}}
=
\indicatorfunction[\{\scaledaggregateclaimssequence \in \randomset \}]
\left( w+\frac{1-w}{\pr(\scaledaggregateclaimssequence\in B_\smallnumber)}\indicatorfunction[\{\scaledaggregateclaimssequence\in B_\smallnumber\}] \right)^{-1}.
\end{equation*}
The output analysis is performed similarly to the Monte Carlo method, i.e.\ we generate $M$ i.i.d.\ replicates of $Z_n$ from $\mathds Q_{\smallnumber,w}$ and we estimate $\pr(\scaledaggregateclaimssequence\in \randomset )$ as the arithmentic mean of the replicates. From Theorem~1 in \cite{bohan2017efficient}, there exists \smallnumber such that the simulation estimator has a bounded relative error.
Hence, the number of simulation runs required to achieve a given accuracy is bounded as $n$ goes to infinity.
For more details of the estimator, we refer the readers to \cite{bohan2017efficient}.

\section*{Acknowledgements}
H.A.\ and E.V.\ acknowledge financial support from the Swiss National Science Foundation Project 200021\_168993. B.C.\ and B.Z.\ are supported by  NWO VICI grant \# 639.033.413 of the Dutch Science Foundation.

\bibliography{ecomor}

\bibliographystyle{abbrv}

\end{document}